\documentclass[final]{siamltex}
\usepackage{amsfonts,amssymb,amsmath} 
\usepackage{booktabs,makecell}
\usepackage{diagbox}
\usepackage{color}
\usepackage{combelow}
\usepackage[breaklinks,colorlinks=true,linkcolor=blue,citecolor=red,backref=page]{hyperref}
\usepackage[framed,numbered,autolinebreaks]{mcode}

\def\proscal<f,g>{\langle\#1,\#2\rangle}
\DeclareMathAlphabet{\itbf}{OML}{cmm}{b}{it}
 \DeclareMathAlphabet\mathbfcal{OMS}{cmsy}{b}{n}

\newcommand{\FF}{\mathbb{F}}
\newcommand{\ZZ}{\mathbb{Z}}
\newcommand{\EE}{\mathbb{E}}
\newcommand{\PP}{\mathbb{P}}
\newcommand{\RR}{\mathbb{R}}

\def\eps{\varepsilon}

\newcommand{\qed}{\hfill $\Box$ \medskip}

\def\bi{{\itbf i}}
\def\bZ{{\itbf Z}}

\def\bz{{\itbf z}}

\def\mus{{\mu_{\rm s}}}
\def\mud{{\mu_{\rm d}}}

\def\boxi{\boldsymbol{\xi}}

\newtheorem{thm}{Theorem}[section]

\newtheorem{prop}[thm]{Proposition}

\newtheorem{remark}[thm]{Remark}

\newcommand{\lm}[1]{{\textcolor{red}{#1}}}

\usepackage[ruled,vlined]{algorithm2e}
\SetKw{KwBy}{by}

\usepackage{color,ulem}
\usepackage{graphicx}%
\usepackage[active]{srcltx}
\usepackage{mathabx}
\usepackage{verbatim}
\usepackage{cancel}
\usepackage{mathtools}
\usepackage{tikz,pgfplots}
\usetikzlibrary{snakes}
\usepackage{cite}
\hypersetup{
  linkcolor  = green!60!black,
  citecolor  = blue!60!black,
  urlcolor   = red!60!black,
  colorlinks = true,
}
 \usepackage{tikz}
 \usetikzlibrary{calc,patterns,
                 decorations.pathmorphing,
                 decorations.markings}
                 
\begin{document}
 
\title{A piecewise deterministic Markov process approach modeling a dry friction problem with noise}

\author{Josselin Garnier\footnotemark[1] \and Ziyu Lu\footnotemark[2]  \and Laurent Mertz\footnotemark[3]}

\maketitle

\author{Josselin
Garnier\thanks{\footnotesize Centre de Math\'ematiques Appliqu\'ees, Ecole Polytechnique, Institut Polytechnique de Paris,
91128 Palaiseau Cedex, France (josselin.garnier@polytechnique.edu)} 
\and Ziyu Lu
\thanks{\footnotesize }
\and Laurent Mertz
\thanks{\footnotesize }
}

\renewcommand{\thefootnote}{\arabic{footnote}}

\begin{abstract}
Understanding and predicting the dynamical properties of systems involving dry friction is a major concern in physics and engineering. It abounds in many mechanical processes, from the sound produced by a violin to the screeching of chalk on a blackboard to human infant crawling dynamics and friction-based locomotion of a multitude of living organisms (snakes, bacteria, scallops) to the displacement of mechanical structures (building, bridges, nuclear plants, massive industrial infrastructures) under earthquakes and beyond. 
Surprisingly, even for low-dimensional systems, the modeling of dry friction in the presence of random forcing has not been elucidated.
In this paper, we propose a piecewise deterministic Markov process approach modeling a system
with dry friction including different coefficients for the static and dynamic forces. 
In this mathematical framework, we derive the corresponding Kolmogorov equations 
to compute statistical quantities of interest related to the distributions of the static (sticked) and dynamic phases. We show ergodicity and provide a representation formula of the stationary measure using independent identically distributed portions of the trajectory (excursions). We also obtain deterministic characterizations of the Laplace transforms of the probability density functions of the durations of the static and dynamic phases. In particular, the analysis of the power spectral density of the velocity reveals a critical value of the noise correlation time below which the correlations of the dynamic behaviors coincide with those of the white noise limit. The existence of such a critical value was already mentioned in the physical literature [Geffert and Just, Phys. Rev. E. {\bf 95} 062111 (2017)].
 \end{abstract}

%

\pagestyle{myheadings}
\thispagestyle{plain}
\markboth{J. Garnier, Z. Lu, and L. Mertz}
{Dry friction problem with noise}

\section{Introduction}
\noindent  
Modeling dry friction is a major concern in physics and engineering. Indeed, it is estimated that $20\%$ of the world's total energy consumption is used to overcome friction \cite{HE2017}. The present work is motivated by the study of the probability distribution of the response of a dry friction model subjected to a certain type of random forces.
To understand the problem, the simplest way is to consider the one-dimensional displacement $U$ of an object (with unit mass) lying on a motionless surface, see Figure~\ref{fig:friction}. The velocity is denoted by $V$ and thus $\dot U = V$. Newton's law implies $\dot V + \FF = b$ where $\FF$ is the force of dry friction and $b$ represents all the other external and internal forces. It is important to emphasize that the force $\FF$ cannot be expressed in terms of a standard function. Below a certain threshold for the applied forces $|b| \leq \mus$ and when $V=0$, the object remains at rest so that we may have $V=0$ in a non-empty time interval (static phase). Otherwise when $|b| > \mus$ or $V\neq 0$ it moves (dynamic phase). Here $\mus>0$ is called static friction coefficient.  \textit{In static phase}, a necessary condition for equilibrium is therefore $\FF = b$. \textit{In dynamic phase}, the force $\FF$ opposes the motion and Coulomb's law implies  $|\FF| = \mud$ where $\mud>0$ is called dynamic friction coefficient. We will assume that $b$ has the form of an internal forcing described by a well behaved function $b(X,V)$, where $X$ is an external forcing that can be random. Here $b$ is real valued but its domain is $\mathbb{R}^{d+1}, d \geq 1$. Indeed $X$ can be multivariate, for instance a $d$-dimensional Ornstein-Uhlenbeck process (see examples in the third section of \cite{cpam}).
In this way, the equation of motion becomes
\begin{equation}
\label{dryfriction}
\dot V + \FF = b(X,V).
\end{equation}
The predictive power of dry friction models that appear in the engineering or physics literature
is generally not supported by a mathematical analysis justifying the well-posedness of the models.
Surprisingly, there is no general mathematical framework for modelling~dry~friction~where~$\mud \leq \mus$.

\begin{figure}[h!]
\begin{center}
\begin{tikzpicture}[scale=0.65]
\coordinate (forcef) at (2.75,0.5);
\coordinate (forceF) at (-2.75,0.5); 
\coordinate (info) at (0,2.5); %
\coordinate (info2) at (0,3.5); %
\draw[dotted,thick] (-4,0) -- (4,0);
\draw[thick] (-1.5,0.05) -- (-1.5,2);
\draw[thick] (1.5,0.05) -- (1.5,2);
\draw[thick] (-1.5,0.05) -- (1.5,0.05);
\draw[thick] (-1.5,2) -- (1.5,2);
\fill[color=gray!50!white] (-1.5,0.05) rectangle (1.5,2);
\draw[->,>=stealth,very thick,black] (1.5,1) -> (4.0,1);
\draw[->,>=stealth,very thick,black] (-1.5,1) -> (-4.0,1);
\node at (forcef) {$b$};
\node at (forceF) {$\FF = \pm \mud$};
\node at (info) {$V \neq 0$ or $V=0, |b| > \mus$ };
\node at (info2) {dynamic phase};
\end{tikzpicture}
\hspace{1cm}
\begin{tikzpicture}[scale=0.65]
\coordinate (forcef) at (2.75,0.5);
\coordinate (forceF) at (-2.75,0.5); 
\coordinate (info) at (0,2.5); %
\coordinate (info2) at (0,3.5); %
\draw[dotted,thick] (-4,0) -- (4,0);
\draw[thick] (-1.5,0.05) -- (-1.5,2);
\draw[thick] (1.5,0.05) -- (1.5,2);
\draw[thick] (-1.5,0.05) -- (1.5,0.05);
\draw[thick] (-1.5,2) -- (1.5,2);
\fill[color=gray!50!white] (-1.5,0.05) rectangle (1.5,2);
\draw[->,>=stealth,very thick,black] (1.5,1) -> (4.0,1);
\draw[->,>=stealth,very thick,black] (-1.5,1) -> (-4.0,1);
\node at (forcef) {$b$};
\node at (forceF) {$\FF = b$};
\node at (info) {$V=0, |b| \leq \mus$ };
\node at (info2) {static phase};
\end{tikzpicture}
\end{center}
\caption{Dynamic (left) and static (right) phases. $\FF$ represents the friction force in response to the applied forces $b$ to the object (shaded area) lying on a motionless surface (dotted line).
\label{fig:friction}}
\end{figure}
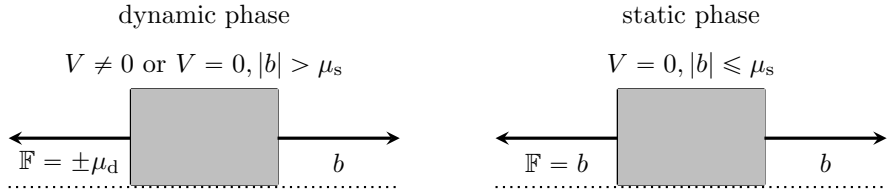
Nonetheless, in some cases, it is possible to justify the well-posedness of the model with an ad-hoc mathematical analysis.
We have in mind the case where $\mud = \mus = \mu$ (in this case we drop the subscript notation ``s" or ``d") and $X$ is a {\color{black} real-valued} deterministic continuous function or $X$ is the continuous solution of a {\color{black} one dimensional} stochastic differential equation. Under such circumstances, the model is well-posed in terms of a differential inclusion (also called multivalued differential equation) \cite{brezis73,marques93,siamStewart}
as follows
\begin{equation}
\label{puredryfriction}
\dot V + \partial \varphi (V) \ni b(X,V) ,
\end{equation}
where {\color{black} $b(x,v) = \mathfrak{b}(v) + x$ with $\mathfrak{b}(v)$ is a Lipschitz function}, $\varphi(v) = \mu |v|$ and $\partial \varphi(v)$ is the subdifferential operator (in the sense of Moreau and Rockafellar \cite{Moreau63,Rockafellar63,HiriartUrruty2001}) 
\begin{equation}
\partial \varphi(v) = 
\left \{
\begin{array}{rcl}
& \{ \mu \textup{sign}(v) \} & \mbox{ if } v \neq 0,\\
& [-\mu,\mu] & \mbox{ if } v = 0.
\end{array}
\right.
\end{equation}
When $X = \dot W$ is a white noise (formal time derivative of a real-valued Wiener process $W$) the framework of stochastic differential inclusions can be used to define the solution \cite{MR3308895}. Numerical techniques to simulate such dry friction systems are proposed in \cite{DontchevLempio92,BAS2013,BER2003}. Moreover, we also have in mind the case where $\mud < \mus$ and $X$ is a continuous function for which the framework of differential inclusions does not allow us to formulate a well-posed problem \cite{BS08}, but an extended variational inequality (EVI) approach can then be used to resolve this issue \cite{BBCMM21}.

\subsection{Review of related literature oriented toward applications}
A model similar to \eqref{dryfriction} 
is studied in \cite{Shaw1986} where the forcing is deterministic and harmonic (sinusoidal external force). The author proposes an exact solution for the dynamic phase. As mentioned by the author, this type of model can be used to describe beating type motions which may occur in turbine blades in the presence of aerodynamical forces. In \cite{BBCMM21}, the EVI framework is rather general as it covers a Lipschitz drift with (at most) linear growth 
and any continuous in time forcing. In particular, it covers \cite{Shaw1986}. The theory is applied to a real structure with real data where the objective is to estimate both the static and dynamic friction coefficients associated with a single bearing point of a bridge. 
The references below discuss models where static friction and dynamic friction coefficients are identical. A slight extension of \eqref{dryfriction} can be used for modeling biolocomotion strategies which are of practical interest in robotics (biomimetism). In \cite{Lauga2013}, the authors consider a system consisting of two bodies at rest on a flat surface and joined by a controllable linkage. The forces are described by Coulomb friction. They demonstrate that friction based locomotion with one degree of freedom is possible. The references below are relevant to soft matter physics and, in contrast with the references above, they involve random forces. In \cite{Goohpattader2009}, the authors investigate experimentally and by simulation the behavior of small objects on a surface subjected to noise and gravity bias. The object velocity is modeled by the same equation as Equation \eqref{puredryfriction} except that $X$ is replaced by a noise of the form $\sqrt{K} \dot W + \bar{\gamma}$, $K>0, \:\bar{\gamma} \in \mathbb{R}$. Their main results are the following. 
They show experimentally and by simulation that the variance of the object displacement grows linearly with time (here the slope is called \textit{diffusivity}) and the stationary average drift velocity can be fitted with a single master curve $\sim \alpha (1+\beta/K)^{-1}$ for some $\alpha,\beta>0$ covering any angle of inclination of the support. Moreover, their experimental study reveals that the diffusivity scales as $\sim K^{1.61}$ which is not too far off from their simulation predicting a scaling $\sim K^{1.74}$. In \cite{Baule2010}, the authors propose a path integral approach to derive analytical expressions for the transition probability of the object’s velocity and the stationary distribution of the work done on the object due to the external force (white noise). From the latter distribution, they obtain a fluctuation relation for the mechanical work fluctuations. In \cite{Goohpattader2010}, the authors investigate experimentally the stochastic behavior of a small solid object on a solid support subject to nonlinear friction when the forces are a combination of a Gaussian white noise and an external constant bias related to gravity. The two models in their paper are written in terms of non smooth Langevin equations. Both equations can be mathematically formulated using differential inclusions. The first equation has the same structure as Equation \eqref{puredryfriction}. However, the difference resides in the random force which is essentially the time derivative of a drifted Brownian motion. Nonetheless, such a dynamics can be obtained from our model when the relaxation time goes to zero. Their second equation can be formulated using a differential inclusion with an oblique sub-gradient \cite{MR3308895}. Inspired by their previous works, they further explore experimentally how rolling of a sphere is affected by Coulomb friction and noise in \cite{Goohpattader2011}. They propose a model which is similar to Equation \eqref{puredryfriction}. The main difference is that the dry friction force is multiplied by a term depending on the noise strength and the velocity. This multiplicative factor models the transition from nonlinear to linear friction. In \cite{Menzel2011}, the formal Fokker Planck equation for both velocity and displacement has been studied. Details on the analysis of the corresponding spectra are reported. In \cite{Touchette2011}, the formal Fokker Planck equation for the velocity has been studied. In \cite{Gnoli2013}, the authors propose a minimal model for a motor where energy is extracted from an equilibrium bath and dissipated only through Coulomb friction. Their model consists of a wheel rotating with an angular velocity around a fixed axis. The wheel is immersed in a fluid and is subject to collision with molecules, viscous drag and Coulomb friction torque. The equation of motion has the same dimension and structure as Equation \eqref{puredryfriction}. The random force is a ``kick" noise which can be seen as the time derivative of a Markov jump process. In \cite{GJ2017}, the closest reference to our present work, the authors consider the case of pure dry friction (\ref{puredryfriction}) with $\mud=\mus=1$ and { \color{black} $\mathfrak{b}(v)=0$} and replace the term $\partial \varphi(v)$ by a smoother term $\sigma_{\epsilon}(v) = \tanh({v}/{\epsilon})$. In this way, they investigate the equation $\dot{V} + \sigma_{\epsilon}(V)  = X$, where $X$ is a Gaussian process with mean zero and covariance function $\EE[X(t)X(s)] = \frac{1}{2\tau}\exp(- \frac{|t-s|}{\tau})$. Then, they apply the unified colored noise approximation (UCNA), previously developed by Jung and H\"{a}nggi \cite{JH1987}, to obtain an approximate expression 
 of the stationary probability density function (pdf) of the process $V$ for any fixed $\epsilon>0$. They then take a formal limit as $\epsilon \to 0$ to obtain a formula for  the probability of sticking (the mass of the singular part of the pdf at $V=0$). This analytic approximation works rather well for small values of $\tau$, but fails for values of order one. It provides, however, valuable insights into the underlying stochastic dynamics. The approach that we propose in this paper is different and has more rigorous theoretical foundations.

\subsection{Our contribution: A piecewise deterministic Markov process approach}
In this paper, we propose a piecewise deterministic Markov process (PDMP) approach to model dry friction as informally presented in \eqref{dryfriction}. 
We consider the case where $X$ takes real values, for higher dimensions the idea remains the same but it is not discussed in this manuscript.
In this approach 1) the external forcing $X$ takes discrete values and it is assumed to be a Markov jump process; 2) given the step-wise constant trajectory $X$, the velocity $V$ satisfies (\ref{dryfriction}).
In this way, the process satisfies a well-posed problem. In this regard, we obtain a solid mathematical framework for deriving the Kolmogorov equations, shown in section \ref{pdmpkolmogorov}, and related tools to compute statistical quantities of interest. 

In the case where $\mud = \mus =\mu$, we show in Proposition \ref{prop:1} that the aforementioned process converges in distribution towards the solution of the differential inclusion  \eqref{puredryfriction} driven by the continuous solution of a stochastic differential equation as the step size in $X$ goes to $0$.

The introduction of the PDMP framework makes it possible to obtain relevant results about the dry friction problem with noise. We obtain the general representation formulas (\ref{eq:represnu}) and (\ref{eq:rep_excursion}) for the stationary distribution of the dry friction process.
The first one makes it possible to compute relevant quantities by solving Kolmogorov equations, while the second one makes it possible to estimate the same quantities by an efficient Monte Carlo method.
We compute dynamical properties in Section \ref{sec:dyn}, such as the power spectral density of the velocity and the distributions of the durations of the sticking and sliding periods.

\section{A semi-discrete Markov process approach model for dry friction}
\label{pdmpkolmogorov}
Ideally we would like to consider an external forcing that is a colored noise $X_t$, that is itself solution of a stochastic differential equation 
\begin{equation}
\label{noiselimit}
\dot X = - {\tau}^{-1}X + \sqrt{2}{\tau}^{-1} \dot W
\end{equation}
where $\dot W_t$ is a white noise and $\tau>0$ is the noise correlation time.
The infinitesimal generator $Q$ of the continuous Markov process $X_t$ has the form
\begin{equation}
\forall f \: \mbox{twice differentiable function}, 
\: Q f =  {\tau}^{-2} f'' - \tau^{-1}xf'. 
\end{equation}
The process $X_t$ is stationary and ergodic and 
its invariant probability distribution is the normal distribution with density $g(x) = \sqrt{\tau}  {e^{-{\tau x^2}/{2}}}/{\sqrt{2 \pi}}$. As $\tau \to 0$, $X$ behaves like $\sqrt{2} \dot{W}$.

In this section, we propose an approximation of $X$ by a pure jump process $X^\delta$ where $\delta >0$ is a small number.
The state space of $X^\delta$ is denoted by $S^\delta = \delta \ZZ \cap [ - L_X^\delta, L_X^\delta]$, with $L_X^\delta \to +\infty$ and $\delta L_X^\delta \to 0$ as $\delta \to 0$. Thus, for any $\delta>0$, it is a finite set of equally $\delta$-spaced points denoted by $\{x_{-N},\ldots,x_N\}$ with $x_{ \pm N} = \pm L_X^\delta$.
We denote the cardinality of $S^\delta$ by $2N+1$, $N=[L_X^\delta \delta^{-1}]$. We denote by $k_{\mus}$ the index such that $x_{k_{\mus}} \leq \mus$ and $x_{k_{\mus}+1} > \mus$. We also have $\{ x_{i}, \: | i | \leq k_\mus \} \subset [-\mus,\mus]$ and if $|i| > k_{\mus}$ then $|x_i|> \mus$. The process $X^\delta$ is a jump Markov process with the infinitesimal generator 
\begin{equation}
\label{def:Qdelta}
Q^\delta f(x) = 2\tau^{-2} \delta^{-2} \left ( \alpha(x) f(x+\delta) - f(x) + (1-\alpha(x)) f(x-\delta) \right ),
\end{equation}
where (assuming $\tau \delta L_X^\delta <2$)
\begin{equation}
\alpha(x) =\left\{
\begin{array}{ll}
\frac{1}{2} \big(1 -\frac{\tau \delta}{2}\big)  & \mbox{ if } |x| < N \delta ,\\
0 &\mbox{ if }  x =N\delta,\\
1 &\mbox{ if }  x =-N\delta.
\end{array}
\right.
\end{equation}
In this context, replacing $X_t$ by $X^\delta_t$, the pure (i.e. when $\mus = \mud = \mu$ and  $\varphi (v) = \mu |v|$) dry friction model is replaced by
\begin{equation}
\label{dryfrictiondelta}
\dot V^\delta + \partial \varphi (V^\delta) \ni { \color{black} \mathfrak{b}(V^\delta) + X^\delta}.
\end{equation}
A motion illustration is shown in Figure \ref{fig:motion}.
	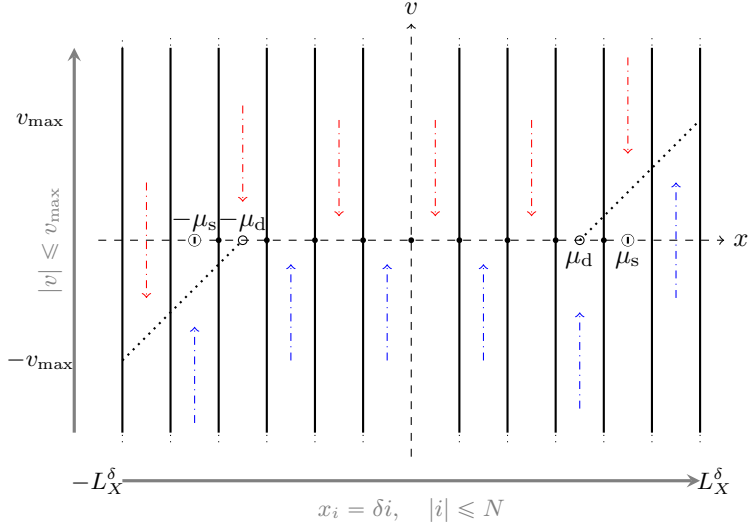
\begin{figure}[h!]
		\centering
		\begin{tikzpicture}[scale=0.64]
		\draw[->,dashed] (-6.5,0) -- (6.5,0) node[right] {$x$};
		\draw[->,dashed] (0,-4.5) -- (0,4.5) node[above] {$v$};
		
		\draw[-,black, thick] (-6,-4) -- (-6,4);
		\draw[-,black, dotted] (-6,-4.2) -- (-6,-4);
		\draw[-,black, dotted] (-6,4) -- (-6,4.2);
		
		\draw[-,black, thick] (-5,-4) -- (-5,4);
		\draw[-,black, dotted] (-5,-4.2) -- (-5,-4);
		\draw[-,black, dotted] (-5,4.2) -- (-5,4);
		
		\draw[-,black, thick] (-4,-4) -- (-4,4);
		\draw[-,black, dotted] (-4,-4.2) -- (-4,-4);
		\draw[-,black, dotted] (-4,4.2) -- (-4,4);
		
		\draw[-,black, thick] (-3,-4) -- (-3,4);
		\draw[-,black, dotted] (-3,-4.2) -- (-3,-4);
		\draw[-,black, dotted] (-3,4.2) -- (-3,4);
		
		\draw[-,black, thick] (-2,-4) -- (-2,4);
		\draw[-,black, dotted] (-2,-4.2) -- (-2,-4);
		\draw[-,black, dotted] (-2,4.2) -- (-2,4);
		
		\draw[-,black, thick] (-1,-4) -- (-1,4);
		\draw[-,black, dotted] (-1,-4.2) -- (-1,-4);
		\draw[-,black, dotted] (-1,4.2) -- (-1,4);
		
			
		\draw[-,black, thick] (6,-4) -- (6,4);
		\draw[-,black, dotted] (6,-4.2) -- (6,-4);
		\draw[-,black, dotted] (6,4.2) -- (6,4);
		
		\draw[-,black, thick] (5,-4) -- (5,4);
		\draw[-,black, dotted] (5,-4.2) -- (5,-4);
		\draw[-,black, dotted] (5,4.2) -- (5,4);
		
		\draw[-,black, thick] (4,-4) -- (4,4);
		\draw[-,black, dotted] (4,-4.2) -- (4,-4);
		\draw[-,black, dotted] (4,4.2) -- (4,4);
		
		\draw[-,black, thick] (3,-4) -- (3,4);
		\draw[-,black, dotted] (3,-4.2) -- (3,-4);
		\draw[-,black, dotted] (3,4.2) -- (3,4);
		
		\draw[-,black, thick] (2,-4) -- (2,4);
		\draw[-,black, dotted] (2,-4.2) -- (2,-4);
		\draw[-,black, dotted] (2,4.2) -- (2,4);
		
		\draw[-,black, thick] (1,-4) -- (1,4);	
		\draw[-,black, dotted] (1,-4.2) -- (1,-4);
		\draw[-,black, dotted] (1,4.2) -- (1,4);	
		 
		

		\foreach \x in {-4,...,4}
		\foreach \y in {0}{
			\fill[black] (\x,\y) circle (0.06cm);
		}
	
	\node at (-7.7,-2.5) {\small $-{v}_{\rm max}$};
	\node at (-7.7,2.5) {\small ${v}_{\rm max}$};
	\draw[->,>=stealth,very thick,gray] (-7,-4) -> (-7,4);
	\node[rotate=90] at (-7.5,0) {\small\textcolor{gray}{ $|v| \leq {v}_{\rm max}$}};
	
	\node at (-6.5,-5) {\small $-L_X^\delta$};
	\node at (6.3,-5) {\small $L_X^\delta$};
	\draw[->,>=stealth,very thick,gray] (-6,-5.) -> (6,-5.);
	\node at (0,-5.6) {\small\textcolor{gray}{$x_i = \delta i, \quad |i| \le N$}};
		
		
		\draw[dotted, thick, scale=1,domain=3.5:6,smooth,variable=\x,black] plot ({\x},{-3.5+\x});
		\draw[dotted, thick, scale=1,domain=-6:-3.5,smooth,variable=\x,black]  plot ({\x},{3.5+\x});
		
		\node at (3.5,-0.4) {$\mud$};
		
                 \draw[black] (3.5,0) circle (0.09cm);
                 \draw[black] (-3.5,0) circle (0.09cm);
                 
                 \draw[black] (4.5,0) circle (0.12cm);
                 \fill[white] (4.5,0) circle (0.12cm);
                 \draw[black] (-4.5,0) circle (0.12cm);
                 \fill[white] (-4.5,0) circle (0.12cm);    		
		
		\node at (-3.5,0.4) {$-\mud$};
		\draw[thick,black] (-3.,-0.07) -- (-3.,0.07);
		\node at (-4.5,0.4) {$-\mus$};
		\draw[thick,black] (-4.5,-0.07) -- (-4.5,0.07);
		
		\node at (4.5,-0.4) {$\mus$};
		\draw[thick,black] (4.5,-0.07) -- (4.5,0.07);
		


		\draw[->,blue, dashdotted] (5.5,-1.2) -- (5.5,1.2);		
		\draw[->,blue, dashdotted] (-2.5,-2.5) -- (-2.5,-0.5);
		\draw[->,blue, dashdotted] (-0.5,-2.5) -- (-0.5,-0.5);
		\draw[->,blue, dashdotted] (1.5,-2.5) -- (1.5,-0.5);
		\draw[->,blue, dashdotted] (3.5,-3.5) -- (3.5,-1.5);

		\draw[->,blue, dashdotted] (-4.5,-3.8) -- (-4.5,-1.8);	
		\draw[->,red, dashdotted] (-5.5,1.2) -- (-5.5,-1.2);

		\draw[->,red, dashdotted] (2.5,2.5) -- (2.5,0.5);
		\draw[->,red, dashdotted] (0.5,2.5) -- (0.5,0.5);
		\draw[->,red, dashdotted] (-1.5,2.5) -- (-1.5,0.5);
		\draw[->,red, dashdotted] (-3.5,2.8) -- (-3.5,0.8);

		\draw[->,red, dashdotted] (4.5,3.8) -- (4.5,1.8);
		
		\end{tikzpicture}
		\caption{Moving directions of $(X^\delta,V^\delta)$ when $x_N>\mus > \mud$ and { \color{black} $\mathfrak{b}(v)=-v$} (so that $v_\textup{max} = x_N-\mud$). \textbf{Dynamic phase:} away from the black points $\{(x_i,0), i=-k_\mus,\ldots, k_\mus\}$, $(X^\delta,V^\delta)$ can move continuously upward and downward respectively and by jumps along the $x$-axis. \textbf{Static phase:}  at the black points $\{(x_i,0), i=-k_\mus,\ldots, k_\mus\}$, $(X^\delta,V^\delta)$ moves only by jumps along the $x$-axis.
		 \label{fig:motion}}
	\end{figure}
It is worth mentioning that Equation \eqref{dryfrictiondelta} does not cover the case $\mus > \mud$. In the latter, the formulation of the dynamics does not involve any subdifferential operator. See Remark \ref{newremark1}.

The process $(X^\delta_t,V^\delta_t)$ is well defined as a c\`adl\`ag (right continuous with left limits \cite{bill}) process on $S^\delta \times \RR$.
It is also possible to interpret the semi-discrete Markov process $(X^\delta_t,V^\delta_t)$ in terms of a Piecewise Deterministic Markov Process (PDMP) and this is the main idea of this paper.
The theory developed for PDMPs then makes it possible to write Kolmogorov equations and use dedicated tools and results.
We introduce the process $\bZ^\delta_t= (X^\delta_t, Y^\delta_t,V^\delta_t)$,
where $(X^\delta_t,V^\delta_t)$ is the process defined here above by (\ref{def:Qdelta}-\ref{dryfrictiondelta}) and we have added the marker $Y^\delta_t = \Theta(X^\delta_t,V^\delta_t)$, with 
\begin{equation}
 \Theta(x,v) = 
\left\{
\begin{array}{ll}
1 & \mbox{ if }v>0 \mbox{ or if } v=0,  \, {\color{black} \mathfrak{b}(0) +} x > \mus,\\ -1& \mbox{ if }v<0 \mbox{ or if } v=0, \, {\color{black} \mathfrak{b}(0) +} x < -\mus,\\ 0 &\mbox{ if }v=0, \, {\color{black} \mathfrak{b}(0) +} x \in [-\mus,\mus].
\end{array}
\right.
\end{equation}
Then $(X^\delta_t,Y^\delta_t)$ is a jump Markov process which takes values in the finite space $S^\delta \times \{-1,0,1\}$ and which has c\`adl\`ag trajectories,
$V^\delta_t$ is a real-valued continuous process, and $(X^\delta_t, Y^\delta_t,V^\delta_t)$ is a Markov process, more exactly a PDMP, whose infinitesimal generator is given below.
The introduction of the marker $Y^\delta_t$ makes it possible to adopt the formalism of PDMPs, with
smooth flows for the continuous process $V^\delta_t$ and jumps of the mode $(X^\delta_t,Y^\delta_t)$ that occur at random times (when $X^\delta_t$ jumps) and at deterministic times when the process hits the boundaries of the state space described below (when $V^\delta_t$ reaches $0$ the dynamics for $V^\delta_t$ changes).

When $\mud=\mus=\mu$ we can establish the connection between this semi-discrete Markov process and the continuous process solution of \eqref{noiselimit}-\eqref{puredryfriction}.
Such a result in the case $\mus>\mud$ is beyond the scope of this paper as the limit system is not clear in this case.  
\begin{prop}
\label{prop:1}
If $\mud=\mus=\mu$ then the random processes $(X^\delta_t,V^\delta_t)$ converge in distribution in the space of the c\`adl\`ag functions to the Markov process $(X_t,V_t)$ which is solution of \eqref{noiselimit}-\eqref{puredryfriction}.
\end{prop}

\begin{proof}
This proposition can be proved in two steps:
one first shows that $(X^\delta_t)$ converges to $(X_t)$ as $\delta \to 0$ by standard diffusion approximation theory, and then one shows that the mapping from $(X^\delta_t)$ to $(V^\delta_t)$ through (\ref{dryfrictiondelta}) is continuous.
The detailed proof is in Appendix \ref{app:prop1}.
\end{proof}
\begin{remark}
\label{newremark1}
An essential ingredient of the proof is the continuity of the mapping $X^\delta \to V^\delta$ through (2.5). 
When $\mud < \mus$, the way to define the mapping does not rely on a monotone maximal multivalued operator \cite{brezis73} (replacing $\partial \varphi$). Such a mapping can be defined using an EVI approach \cite{BBCMM21}, however it is not continuous in general (shown below).  This explains why convergence of the system holds only for $\mud = \mus$. In the EVI framework, for any continuous function $x(\cdot)$, the mapping produces a function $v(\cdot)$ for which the phases $\dot v \pm \mud \text{sign}(v) = {\color{black} \mathfrak{b}(v) + x}$ and ${\color{black}|\mathfrak{b}(0) + x|} \leq \mus$ occur when $\pm v>0$ and $v=0$, respectively, on non-empty time intervals.
To see that such a mapping $x \mapsto v(x)$ is not continuous in general, 
we consider for instance $b(x,v) = x$ and $v(0)=0$. If $\forall t \geq 0, x(t) = \mus$ then $\forall t \geq 0, v(t) = 0$ whereas if $\forall t \geq 0, x^\epsilon(t) = \mus + (\epsilon - t)\mathbf{1}_{t \in [0,\epsilon]}$ then 
$\forall t \geq 0, v^\epsilon(t) = (\mus-\mud) t + \epsilon (\epsilon \wedge t) \left ( 1- \frac{t \wedge \epsilon}{2} \right ) $.
Therefore it is clear that $\lim \limits_{\epsilon \to 0} \| x - x^\epsilon \| = 0$ but $\lim \limits_{\epsilon \to 0} \| v - v^\epsilon \| = (\mus-\mud)T$ where $\| \cdot \|$ is the max norm on $[0,T]$.\\ 
\end{remark}
The state space of the process $\bZ^\delta$ is
\begin{equation}
E = \bigcup_{ (x,y) \in 
\mathbb{S}^\delta 
} E_{x,y}, \quad \quad E_{x,y}= \{(x,y) \}\times H_{x,y},
\end{equation}
where 
$\mathbb{S}^\delta = \{ x_{-N},\ldots ,x_{-k_\mus-1}\} \times \{-1,1\} \cup \{ x_{-k_\mus},\ldots ,x_{k_\mus}\} \times \{-1,0,1\} \cup \{ x_{k_\mus+1},\ldots ,x_{N}\} \times \{-1,1\}$,
$H_{x,y} = (-\infty,0)$ if $(x,y) \in
 \{ x_{-k_{\mus}} ,\ldots, x_N \} \times \{-1\}$,
$H_{x,y}=(0,+\infty)$ if $(x,y) \in \{ x_{-N},\ldots, x_{k_{\mus}} \} \times \{1\}$,
and $H_{x,y}=\RR$ otherwise.
Let ${\mathcal E}$ denote the class of measurable sets in $E$:
\begin{equation}
{\mathcal E} = \sigma \big( A_{x,y} , \, A_{x,y} \in {\mathcal E}_{x,y}, (x,y) \in 
\mathbb{S}^\delta 
\big) ,
\end{equation}
where ${\mathcal E}_{x,y}$ denotes the Borel sets of $E_{x,y}$.

The Markov evolution of $\bZ^\delta$ is determined by the following objects:\\
- the real-valued and smooth vectors fields $B(\bz)$,
$\bz= (x,y,v)$, given by
\begin{equation}
B(x,-1,v) = \mud+{\color{black} \mathfrak{b}(v) + x },\quad \quad
B(x,0,v) = 0 ,\quad \quad 
B(x,1,v) = -\mud+{\color{black} \mathfrak{b}(v) + x },
\end{equation}
- the function {\color{black} $v \mapsto \mathfrak{b}(v)$} is Lipschitz continuous and satisfies {\color{black}$\mathfrak{b}(-v)=-\mathfrak{b}(v)$}; {\color{black} in particular $\mathfrak{b}(0) = 0$;} the function $v \in [0,+\infty) \mapsto {\color{black} \mathfrak{b}(v)}$ is decreasing from {\color{black} $0$} to $-\infty$ (we may think for instance that $\mathfrak{b}(v)=-v$),\\
- the constant rate function $\Lambda= 2 \tau^{-2} \delta^{-2}$,\\
- the probability transition measure ${\mathcal Q} : {\mathcal E} \times \overline{E}  \to [0,1]$ is discrete because $V^\delta$ does not jump:

${\cal Q}f(x,y,v) = \sum_{(x',y') \in 
\mathbb{S}^\delta 
}
{\cal Q}\big( (x',y',v);(x,y,v)\big) f(x',y',v)$  and it is given by
\begin{subequations}
\label{eq:defcalQ}
    \begin{align}
&
\forall y \in \{-1,1\}, \: 
\forall x \in \{ x_{-N}, \ldots, x_{-k_{\mus}-1} \}, \:
 {\mathcal Q} \big(  (x,-1,0) ; (x,y,0) ) =1,\\
&
\forall y \in \{-1,1\}, \: 
\forall x \in \{ x_{k_{\mus}+1}, \ldots, x_N \}, \:
 {\mathcal Q} \big(  (x,1,0) ; (x,y,0) ) =1,\\
&
\forall y \in \{-1,1\}, \: 
\forall x \in \{ x_{-k_{\mus}}, \ldots, x_{k_{\mus}} \}, \:
{\mathcal Q} \big(  (x,0,0) ; (x,y,0) ) =1,\\[2mm]
&
\forall {\color{black}x} \in \{ {\color{black}x}_{-k_{\mus}+1}, \ldots, {\color{black}x}_{k_{\mus}-1} \}, \:
{\mathcal Q} \big(  ({\color{black}x}+\delta,0,0) ; ({\color{black}x},0,0) ) = \alpha({\color{black}x}),\\
&
\forall {\color{black}x} \in \{ {\color{black}x}_{-k_{\mus}+1}, \ldots, {\color{black}x}_{k_{\mus}-1} \}, \:
{\mathcal Q} \big(  ({\color{black}x}-\delta,0,0) ; ({\color{black}x},0,0) ) = 1-\alpha({\color{black}x}),\\[2mm]
&
{\mathcal Q} \big( ({\color{black}x}_{-k_{\mus}+1},0,0) ; ({\color{black}x}_{-k_{\mus}},0,0) ) =\alpha({\color{black}x}),    \\
&
{\mathcal Q} \big( ({\color{black}x}_{k_{\mus}-1},0,0) ; ({\color{black}x}_{k_{\mus}},0,0) ) =1-\alpha({\color{black}x}),    \\
&
{\mathcal Q} \big( ({\color{black}x}_{-k_{\mus}-1},-1,0) ; ({\color{black}x}_{-k_{\mus}},0,0) ) =1-\alpha({\color{black}x}) , \\
&
{\mathcal Q} \big( ({\color{black}x}_{k_{\mus}+1},1,0) ; ({\color{black}x}_{k_{\mus}},0,0) ) =\alpha({\color{black}x}),    \\
&
\forall ({\color{black}x},{\color{black}y},v) \in 
E,   \: v \neq 0, \:
{\mathcal Q} \big( ({\color{black}x}+\delta,{\color{black}y},v) ; ({\color{black}x},{\color{black}y},v) ) =\alpha({\color{black}x}),\\
&
 \forall ({\color{black}x},{\color{black}y},v) \in 
E,  \:v \neq 0, \: 
{\mathcal Q} \big( ({\color{black}x}-\delta,{\color{black}y},v) ; ({\color{black}x},{\color{black}y},v) ) =1-\alpha({\color{black}x}) .\end{align}
\end{subequations}

In Eq.~(\ref{eq:defcalQ}):\\
- The first three lines $(a$-$c)$ describe the jumps of the modes when the process $\bZ^\delta$ reaches the boundaries of the domain $\partial E$:
\begin{equation}
\partial E =  \{ {\color{black}x}_{-k_{\mus}},\ldots, {\color{black}x}_N \} \times \{-1\} \times \{0\} \bigcup
\{ {\color{black}x}_{-N},\ldots, {\color{black}x}_{k_{\mus}} \} \times \{1\} \times \{0\}.
\end{equation}
When the process $\bZ^\delta$ reaches $({\color{black}x}_k, -1,0)$ for $k \in \{ k_{\mus}+1, \ldots , N\}$,
it jumps to $({\color{black}x}_k,1,0)$.
When the process $\bZ^\delta$ reaches $({\color{black}x}_k, -1,0)$ for $k \in \{ -k_{\mus}, \ldots, k_{\mus} \}$, it jumps to $({\color{black}x}_k,0,0)$.
These jumps represent the transitions from the dynamic phase with negative velocity (mode ${\color{black}y}=-1$) to the dynamic phase with positive velocity (mode ${\color{black}y}=1$) and to the static phase (mode ${\color{black}y}=0$). Similarly, when the process $\bZ^\delta$ reaches $({\color{black}x}_k, 1,0)$ for $k \in \{ -N, \ldots ,-k_{\mus}-1 \}$,
it jumps to $({\color{black}x}_k,-1,0)$.
When the process $\bZ^\delta$ reaches $({\color{black}x}_k,1,0)$ for $k \in \{ -k_{\mus}, \ldots, k_{\mus} \}$, it jumps to $({\color{black}x}_k,0,0)$.
These jumps represent the transitions from the dynamic phase with positive velocity (mode ${\color{black}y}=1$) to the dynamic phase with negative velocity (mode ${\color{black}y}=-1$) and to the static phase (mode ${\color{black}y}=0$). \\
- The following lines $(d$-$k)$
describe the jumps of the modes that are triggered by the random clock of the driving noise {\color{black}$X^\delta_t$}.
The lines $(d$-$g)$ describe the jumps from the static phase to itself, the lines $(h$-$i)$ describe the jumps from the static phase to the dynamic phase and the lines
$(j$-$k)$ describe the jumps from the dynamic phase to itself.
Note in particular that lines $(h$-$i)$ describe how the process at the border of the static domain at $({\color{black}x}_{\pm k_\mus},0,0)$ can escape the static domain by a jump of {\color{black}$X_t^\delta$} which allows the process to pull itself out of the sticked phase.

We denote $v_\textup{max}= \inf \{ v \geq 0,\, { \color{black} x_N + \mathfrak{b}(v)} - \mud \leq 0\}$. 
We denote by $\Phi_{{\color{black}x},{\color{black}y}}(t,v)$ the flow solution of 
\begin{equation}
\partial_t \Phi_{{\color{black}x},{\color{black}y}}(t,v) = B( {\color{black}x},{\color{black}y},\Phi_{{\color{black}x},{\color{black}y}}(t,v)),
\quad\quad
\Phi_{{\color{black}x},{\color{black}y}}(t=0,v)=v.
\end{equation}
For $\bz = ({\color{black}x},{\color{black}y},v)\in E$, we denote by $T^*(\bz)$ the hitting time of the boundary $\partial {\color{black}H}_{{\color{black}x},{\color{black}y}}$ by $\Phi_{{\color{black}x},{\color{black}y}}(t,v)$.
{\color{black} If ${\color{black}H}_{{\color{black}x},{\color{black}y}} = \RR$ then $T^*(\bz)=+\infty$; otherwise, $\partial {\color{black}H}_{{\color{black}x},{\color{black}y}} =\{0\}$ and this happens only if $\bz \in \{x_{-k_\mus},\ldots,x_N\}\times \{-1\} \times (-\infty,0) \cup \{ x_{-N},\ldots,x_{k_\mus}\} \times \{1\} \times(0,+\infty)$. For such a $\bz$:}
\begin{equation}
T^*(\bz) = 
\inf \big\{ t>0 , \Phi_{{\color{black}x},{\color{black}y}}(t,v) =0 \big\} ,
\end{equation}
with the convention $\inf \emptyset =+\infty$.
If {\color{black} $\mathfrak{b}(v)=-v$}, then the flow $\Phi_{{\color{black}x},{\color{black}y}}(t,v)$ has an explicit expression and we have: 
$$
T^*(\bz) = 
\left\{
\begin{array}{ll}
 -\log \frac{{\color{black}x}-\mud}{{\color{black}x} - \mud-v} 
&
\mbox{ if } v>0, \, {\color{black}x}-\mud<0, \\
 -\log \frac{{\color{black}x}+\mud}{{\color{black}x} + \mud-v} 
 &
\mbox{ if } v<0,  \, {\color{black}x}+\mud>0 , \\
+\infty&
\mbox{ otherwise}.
\end{array}
\right.
$$
For any $\bz \in E$, we define the survivor function $F_\bz$:
\begin{equation}
F_\bz(t) = 
{\bf 1}_{(-\infty,0)}(t)+
\exp (- \Lambda t) {\bf 1}_{[0,T^*(\bz) )}(t) .
\end{equation}
The Markov process $\bZ^\delta$ starting from $\bz_0 = ({\color{black}x}_0,{\color{black}y}_0,v_0) \in E$ is defined as follows.\\
1) Generate a random variable $T_1$ such that $\PP(T_1>t)=F_{\bz_0}(t)$.
Generate a random variable $\bz_1=({\color{black}x}_1,{\color{black}y}_1,v_1)$ 
with distribution ${\mathcal Q}(\cdot ; {\color{black}x}_0,{\color{black}y}_0,\Phi_{{\color{black}x}_0,{\color{black}y}_0}(T_1,v_0))$.
The trajectory of $\bZ^\delta_t$ for $t \in [0,T_1]$ is given by 
\begin{equation}
\bZ^\delta_t = \left\{
\begin{array}{ll}
({\color{black}x}_0,{\color{black}y}_0 , \Phi_{{\color{black}x}_0,{\color{black}y}_0}(t,v_0)) & \mbox{ if } 0\leq t <T_1,\\
({\color{black}x}_1,{\color{black}y}_1,v_1) & \mbox{ if } t =T_1 .
\end{array}
\right.
\end{equation}
2) Starting from $\bZ_{T_1}^\delta =\bz_1$, generate the next inter-jump time $T_2-T_1$ 
such that $\PP(T_2-T_1>t)=F_{\bz_1}(t)$ and the post-jump location $\bz_2=({\color{black}x}_2,{\color{black}y}_2,v_2)$
has distribution  ${\mathcal Q}(\cdot ; {\color{black}x}_1,{\color{black}y}_1,\Phi_{{\color{black}x}_1,{\color{black}y}_1}(T_2-T_1,v_1))$.
The trajectory of $\bZ^\delta_t$ for $t \in [T_1,T_2]$ is given by 
\begin{equation}
\bZ^\delta_t = \left\{
\begin{array}{ll}
({\color{black}x}_1,{\color{black}y}_1 , \Phi_{{\color{black}x}_1,{\color{black}y}_1}(t-T_1,v_1)) & \mbox{ if } T_1\leq t < T_2,\\
({\color{black}x}_2,{\color{black}y}_2,v_2) & \mbox{ if } t =T_2.
\end{array}
\right.
\end{equation}
3) Iterate. This gives a piecewise deterministic trajectory $\bZ^\delta_t$ with jump times
$T_j$, $j \geq 1$.

The process $\bZ^\delta_t = {\color{black}(X^\delta_t,Y^\delta_t,V^\delta_t)}$ is a PDMP 
as introduced by \cite{davis84} and {\color{black}$(X^\delta_t, V^\delta_t)$} follows 
the random dynamics (\ref{def:Qdelta}-\ref{dryfrictiondelta}).
We can then use the theory and simulation methods developed for PDMPs described in \cite{davis84,saporta}.
{ \color{black}
Here are two pseudocodes summarizing the simulation method.\\ 
\begin{algorithm}[H]
\SetAlgoLined
\KwResult{Simulation of  $\left \{\bZ_{T_k}^\delta = (X^\delta,Y^\delta,V^\delta)_{T_k} 
\mbox{ where } \: k \geq 0 \: \mbox{ and }
\: T_k \leq t_f
\right \}$ from initial state $(x,y,v)$.
}
 $T  = 0, 
 \: X  = x, 
 \: Y  = y,
 \: V = v$\;
\While{$T  \leq t_f$}{
$(X',Y',V')=(X,Y,V), \: T' = T$\;
$\mathfrak{u}$ = uniform(.), 
$\delta T'  = \min \left ( - {\log(\mathfrak{u})}/{\Lambda}, 
T^*(X',Y',V')  \right )$\;
$(X,Y,V) = J(X',Y';\Phi(X',Y',V'; \delta T')), \:
T= T' + \delta T'$\;
(At each step $(X,Y,V)$ and $T$ are $(X^\delta,Y^\delta,V^\delta)_{T_k}$ and $T_k$).\\
}
\caption{PDMP simulation for dry friction.}
\end{algorithm}

\begin{algorithm}[H]
\SetAlgoLined
\KwResult{$(X,Y,{v})=J(x,y,{v})$}
$\alpha= \frac{1}{2} \left ( 1- \frac{ \tau \delta x}{2} \right ) \mathbf{1}_{\{ |x| < L_X^\delta \}}
+ (1-x^{-1} \max(x,0) )  \mathbf{1}_{\{ |x| = L_X^\delta \}}$\;
\eIf{$(|y|=1) \wedge (xy \leq x_{k_{\mus}}) \wedge (v=0)$}{$X=x$, $Y = -y\mathbf{1}_{\{ xy \leq x_{-(k_{\mus}+1)}\}}$\;}{
$\mathfrak{u}$ = uniform(.), $X=x+\delta \left ( \mathbf{1}_{\{ \mathfrak{u} \leq \alpha \}} - \mathbf{1}_{\{ \mathfrak{u} > \alpha \}} \right )$\; $Y = \mathbf{1}_{\{ y=0, |x| = x_{k_{\mus}} \}} \big( \mathbf{1}_{\{ x = x_{k_{\mus},  \mathfrak{u} \leq \alpha \} }} - \mathbf{1}_{\{ x = -x_{k_{\mus}},  \mathfrak{u} >\alpha \}} \big)$.
}
\caption{Simulation of a jump from $(x,y,{v})$}
\end{algorithm}
}

The Markov process $\bZ^\delta_t$ is irreducible on $E'$, with
\begin{equation}
\label{eq:defEp}
E'= S^\delta\times \{-1\}\times [-v_\textup{max},0] \bigcup
S^\delta\times \{1\}\times [0,+v_\textup{max}]
\bigcup
\{{\color{black}x}_{-k_{\mus}},\ldots,{\color{black}x}_{k_{\mus}}\} \times\{0\}\times \{0\} .
\end{equation}
More exactly, starting from any $({\color{black}x},{\color{black}y},v)\in E$, the Markov process $\bZ^\delta_t$ reaches $E'$ in finite time and it remains in $E'$ after that time.

The domain of the generator ${\cal L}$  of the process $\bZ^\delta_t$  contains the functions $f$ that are smooth and bounded in $v$ and that satisfy the boundary condition:
\begin{equation}
\label{eq:bcpdmp}
\forall \bz =({\color{black}x},{\color{black}y},v)\in \partial E,
\quad
f(\bz) =  \sum_{({\color{black}x}',{\color{black}y}')  \in 
\mathbb{S}^\delta 
} 
f({\color{black}x}',{\color{black}y}',v)  {\mathcal Q}(({\color{black}x}',{\color{black}y}',v);\bz)  .
\end{equation}
For those functions we have \cite[Theorem 5.5]{davis84}
\begin{align}
{\mathcal L} f(\bz) 
=
 B(\bz) \partial_v f(\bz) +\Lambda \sum_{({\color{black}x}',{\color{black}y}')  \in 
 \mathbb{S}^\delta 
 } 
[f({\color{black}x}',{\color{black}y}',v) -f(\bz)] {\mathcal Q}(({\color{black}x}',{\color{black}y}',v);\bz)  .
\end{align}
More exactly, the domain of the generator consists of the functions $f$ that satisfy the continuity condition:  for all $\bz =({\color{black}x},{\color{black}y},v) \in E$,
$$
t\mapsto f({\color{black}x},{\color{black}y},\Phi_{{\color{black}x},{\color{black}y}}(t,v)) \mbox{ is absolutely continuous for } t \in [0,T^*(\bz)),
$$
an integrability condition (fulfilled when $f$ is bounded), and the boundary condition (\ref{eq:bcpdmp}) \cite[Theorem 5.5]{davis84}.
The boundary condition (\ref{eq:bcpdmp}) can be written more explicitly as
{\color{black}
\begin{equation}
\left\{
\begin{array}{l}
\displaystyle  f({\color{black}x},-1,0) = f({\color{black}x},0,0)=f({\color{black}x},1,0),\quad \forall {\color{black}x} \in \{x_{-k_\mus},\ldots, x_{k_\mus}\} ,\\
\displaystyle  f({\color{black}x},-1,0) =f({\color{black}x},1,0),\quad \forall {\color{black}x} \in \{x_{-N},\ldots,x_{-k_\mus-1} \}\cup\{ x_{k_\mus+1}, \ldots, x_N\}   .
\end{array}
\right.
  \label{eq:bcpdmp2}
\end{equation}
}

As a consequence, for a smooth and bounded function $f(\bz)$ that satisfies (\ref{eq:bcpdmp2}),
we have
\begin{equation}
\EE[ f(\bZ^\delta_t) | \bZ^\delta_0=\bz] = F(0,\bz;t),
\end{equation}
where $(s,\bz) \mapsto F(s,\bz;t)$ is the solution of the backward Kolmogorov equation
\begin{equation}
\label{eq:kolmFs}
\partial_s F +{\mathcal L} F=0 \: \mbox{in} \: E \: \mbox{for} \: s\in (0,t),
\end{equation}
with the boundary condition (\ref{eq:bcpdmp}) on $\partial E$ for $s\in (0,t)$, and the terminal condition $F(s=t,\bz ; t) = f(\bz)$.

\section{Ergodicity and stationary state}
Given $f$ a function in the domain of ${\cal L}$, we consider the function 
\begin{equation}
u_\lambda ({\color{black}x},{\color{black}y},v;f) = \mathbb{E}_{({\color{black}x},{\color{black}y},v)}\Big[  \int_0^\infty e^{-\lambda s} f({\color{black}X^\delta_s ,Y^\delta_s , V^\delta_s}) \textup{d} s \Big].
\end{equation}
From the theory of Markov processes \cite{MR0838085,davis84}, 
it satisfies the equation
\begin{equation}
\label{plambda}
\lambda u_\lambda - {\cal L} u_\lambda = f \: \mbox{ in } \: E.
\end{equation}
Let us introduce the stopping times
{\color{black}
\begin{align}
\hat{\tau}_n  &= \inf \big\{ t \geq \tau_{n-1},  \,  {\color{black}V^\delta_t} = 0 \mbox{ and } | {\color{black}X^\delta_t} | \leq \mus \big\}, \quad n \geq 1, \\
\tau_n &= \inf \big\{ t \geq \hat{\tau}_{n} , \,  ({\color{black}X^\delta_t ,Y^\delta_t , V^\delta_t}) \in \{  \mathfrak{s}_-,\mathfrak{s}_+\}  \big\} ,\quad n\geq 1,
\end{align}
where $\mathfrak{s}_\pm = (\pm {\color{black}x}_{k_{\mus}+1},\pm1,0)$ and $\tau_0=0$.} The two points  $\mathfrak{s}_\pm$ are the two possible exit points of the static phase.
For $n\geq 1$, 
the stopping times $\hat{\tau}_n$ and $\tau_n$ represent respectively the entry and exit times of the $n$-th static phases which are defined as the time intervals when $({\color{black}X^\delta_t ,Y^\delta_t , V^\delta_t}) \in D^0$, $D^0=\{({\color{black}x}_k,0,0), k=-k_{\mus},\ldots,k_{\mus}\}$.
The recurrence and ergodicity of the Markov process is a consequence of the following proposition (see Appendix \ref{app:a} for the proof).

\begin{prop}
\label{prop:31}%
We have
$\EE_{\mathfrak{s}_+} \big[ {\tau_1} \big] < \infty$.
\end{prop}

We propose below a representation formula for the stationary measure of the process $({\color{black}X^\delta_t ,Y^\delta_t , V^\delta_t})$ that
is based on the functions ${h}^\pm_\lambda$ and $w_\lambda$ that are defined by 
\begin{align}
\label{eq:defpiw}
{h}_\lambda^\pm({\color{black}x},{\color{black}y},v) &= \mathbb{E}_{({\color{black}x},{\color{black}y},v)} \big[ e^{-\lambda \tau_1} \mathbf{1}_{ \{ ({\color{black}X^\delta_{\tau_1},Y^\delta_{\tau_1}, V^\delta_{\tau_1} }) = \mathfrak{s}_\pm \} } \big],\\
w_\lambda({\color{black}x},{\color{black}y},v;f) &= \mathbb{E}_{({\color{black}x},{\color{black}y},v)} \Big[ \int_{0}^{\tau_1} e^{-\lambda s} f({\color{black}X^\delta_s ,Y^\delta_s , V^\delta_s}) \textup{d} s \Big]  ,
\label{eq:defwlambda}
\end{align}
where $f$ is  a bounded function. The functions ${h}^\pm_\lambda$ and $w_\lambda$ can be computed as explained in the following proposition.

\begin{prop}
Let us introduce two absorbing states $\mathfrak{s}_\pm'$ and a modified kernel ${\cal L}'$ such that 
\begin{equation}
{\cal L}'f (\bz)=
 B(\bz) \partial_v f(\bz) + { \color{black} \Lambda} \sum_{({\color{black}x}',{\color{black}y}')  \in 
 \mathbb{S}^\delta 
 } 
[f({\color{black}x}',{\color{black}y}',v) -f(\bz)] {\mathcal Q} (({\color{black}x}',{\color{black}y}',v);\bz)
\end{equation}
 for all $\bz \in \overline{E} \backslash \{ ({\color{black}x}_{-k_{\mus}},0,0) , ({\color{black}x}_{k_{\mus}},0,0)  \}$,
${\cal L}'f (\mathfrak{s}_\pm')=0$, and 
\begin{align}
\nonumber
{\cal L'} f \big( ({\color{black}x}_{k_{\mus}},0,0) \big)
=
& \alpha({\color{black}x}_{k_{\mus}}) \big[ f(\mathfrak{s}_+') - f \big( ({\color{black}x}_{k_{\mus}},0,0) \big) \big]\\
& + \big(1-\alpha({\color{black}x}_{k_{\mus}}) \big) \big[ f\big(({\color{black}x}_{k_{\mus}-1},0,0)\big) - f \big( ({\color{black}x}_{k_{\mus}},0,0) \big)   \big],\\
\nonumber
 {\cal L'} f \big( ({\color{black}x}_{-k_{\mus}},0,0) \big) = & 
\alpha({\color{black}x}_{-k_{\mus}}) \big[ f\big(({\color{black}x}_{-k_{\mus}+1},0,0)\big) - f \big( ({\color{black}x}_{-k_{\mus}},0,0) \big) \big]\\  
& +
\big(1-\alpha({\color{black}x}_{-k_{\mus}}) \big) \big[ f(\mathfrak{s}_-')- f \big( ({\color{black}x}_{-k_{\mus}},0,0) \big)   \big] .
\end{align}
The functions ${h}_\lambda^\pm$ and $w_\lambda(\cdot;f)$ defined by (\ref{eq:defpiw}) and (\ref{eq:defwlambda}) satisfy
\begin{align}
& \lambda {h}_\lambda^+ - {\cal L}' {h}_\lambda^+ = 0 \: \mbox{ in } \: E,
\quad \quad {h}_\lambda^+(\mathfrak{s}_+') = 1, \: {h}_\lambda^+(\mathfrak{s}_-') = 0 , \\
& \lambda {h}_\lambda^- - {\cal L}' {h}_\lambda^- = 0 
\: \mbox{ in } \: E,
\quad \quad {h}_\lambda^-(\mathfrak{s}_+') = 0, \: {h}_\lambda^-(\mathfrak{s}_-') = 1 ,
\end{align}
and
\begin{equation}
\label{def:wlambda}
\lambda w_\lambda(\cdot;f) - {\cal L}' w_\lambda(\cdot;f) = f \: \mbox{ in } \: E, 
\quad \quad w_\lambda(\mathfrak{s}_\pm';f) = 0.
\end{equation}
\end{prop}
\begin{proof}
A trajectory of {\color{black} $(X^\delta,Y^\delta,V^\delta)$} can reach $\mathfrak{s}_+$ by different ways. Indeed, just before the time at which $\mathfrak{s}_+$ is reached, the process can be in a static phase at the point $({\color{black}x}_{k_{\mus}},0,0)$ (and it jumps from the static phase to the dynamic phase) or else it can be in a dynamic phase 
at $( {\color{black}x}_{k_{\mus}+1},-1,v)$, $v<0$, and it jumps to $\mathfrak{s}_+$ when ${\color{black}V}^\delta_t$ reaches $0$.
This comment motivates the introduction of the two absorbing states $\mathfrak{s}_\pm'$.
The Markov process {\color{black}$(X^{\delta,\prime},Y^{\delta,\prime},V^{\delta,\prime})$} with the modified kernel ${\cal L}'$ follows the same dynamics as the original process {\color{black}$(X^\delta,Y^\delta,V^\delta)$}, except for one aspect: when the original process is at $({\color{black}x}_{k_{\mus}},0,0)$ (resp. $({\color{black}x}_{-k_{\mus}},0,0)$) and jumps to the east (resp. to the west), it jumps to $\mathfrak{s}_+$
(resp. $\mathfrak{s}_-$); when the modified process is at $({\color{black}x}_{k_{\mus}},0,0)$ (resp. $({\color{black}x}_{-k_{\mus}},0,0)$) and jumps to the east (resp. to the west), it jumps to $\mathfrak{s}_+'$ (resp. $\mathfrak{s}_-'$) and does not move anymore. 
As a consequence,
we have for any $({\color{black}x,y},v)\in E$:
\begin{align}
{h}_\lambda^\pm({\color{black}x},{\color{black}y},v) &= \mathbb{E}_{({\color{black}x},{\color{black}y},v)} \big[ e^{-\lambda \tau_1^\prime} \mathbf{1}_{ ({\color{black}X}^{\delta,\prime}_{\tau_1^\prime},{\color{black}Y}^{\delta,\prime}_{\tau_1^\prime},{\color{black}V}^{\delta,\prime}_{\tau_1^\prime}) = \mathfrak{s}_\pm^\prime } \big] ,\\
w_\lambda({\color{black}x},{\color{black}y},v;f) &= \mathbb{E}_{({\color{black}x},{\color{black}y},v)} \Big[ \int_0^{\tau_1^\prime} e^{-\lambda s} f({\color{black}X}^{\delta,\prime}_s,{\color{black}Y}^{\delta,\prime}_s,{\color{black}V}^{\delta,\prime}_s) \textup{d} s \Big]  ,
\end{align}
where 
$$
\tau_1^\prime = \inf \big\{t\geq 0, ({\color{black}X}_t^{\delta,\prime},{\color{black}Y}_t^{\delta,\prime},{\color{black}V}_t^{\delta,\prime}) \in \{ \mathfrak{s}_-^\prime,\mathfrak{s}_+^\prime\}\big\} ,
$$
and the statement of the proposition follows immediately.
\end{proof}

The function $w_\lambda(\cdot;1)$ is $w_\lambda(\cdot;f)$ when $f=1$. 
The following proposition is inspired from \cite{MR2922084} and is proved in Appendix \ref{app:b}.
\begin{prop}
\label{prop:32}
We have the representation formula 
\begin{align}
\label{repformula}
u_\lambda({\color{black}x},{\color{black}y},v;f)  = & w_\lambda({\color{black}x},{\color{black}y},v;f) - {\pi}_\lambda(f)w_\lambda({\color{black}x},{\color{black}y},v;1) \nonumber \\ 
& + \mu_\lambda(f) \big({h}_\lambda^+ ({\color{black}x},{\color{black}y},v) - {h}_\lambda^-({\color{black}x},{\color{black}y},v )\big) + \frac{{\pi}_\lambda(f)}{\lambda} ,
\end{align}
for all $({\color{black}x},{\color{black}y},v)\in E$,
 where 
\begin{equation}
{\pi}_\lambda(f) = \frac{w_\lambda(\mathfrak{s}_+;f)+w_\lambda(\mathfrak{s}_-;f) }{2 w_\lambda(\mathfrak{s}_+;1)}
\:
\mbox{ and }
\:
\mu_\lambda(f) = \frac{w_\lambda(\mathfrak{s}_+;f)- w_\lambda(\mathfrak{s}_-;f)}{2 (1-{h}_\lambda^+(\mathfrak{s}_+) + {h}_\lambda^-(\mathfrak{s}_+))}.
\end{equation}
We have the following characterization for the stationary distribution of 
the process $({\color{black}X}_t^\delta,{\color{black}Y}_t^\delta,{\color{black}V}_t^\delta)$:
\begin{equation}
{\pi} (f) = \lim_{\lambda \to 0} \lambda u_\lambda({\color{black}x},{\color{black}y},v;f) = 
\frac{w_0(\mathfrak{s}_+;f)+w_0(\mathfrak{s}_-;f)}{2 w_0(\mathfrak{s}_+;1)},
\label{eq:represnu}
\end{equation}
for any bounded function $f$.
\end{prop}

\section{Dynamical properties}
\label{sec:dyn}

\subsection{Excursions}

The random variables $(\tau_{n+1}-\tau_{n})_{n\geq 0}$ are integrable, independent and identically distributed.
This follows from the strong Markov property, from the fact that $\tau_1$ is integrable under $\PP_{\mathfrak{s}_+}$ (we have $\tau_0=0$ a.s. under $\PP_{\mathfrak{s}_+}$), and from the symmetry of the system which implies that the distributions of 
$\tau_1$ starting from $\mathfrak{s}_+$ and from  $\mathfrak{s}_-$ are identical.

For any bounded (in fact, ${\pi}$-integrable) function $f$, ${\pi}(f)$ can be computed by (\ref{eq:represnu}).
By ergodicity we also have
\begin{equation}
\label{eq:rep_excursion}{\pi} (f) = \frac{1}{\EE_{\mathfrak{s}_+} [{\tau_1}]} \EE_{\mathfrak{s}_+} \Big[ \int_0^{{\tau_1}} 
f_{\rm even}({\color{black}X^\delta_s,Y^\delta_s,V^\delta_s}) \textup{d} s\Big]  ,
\end{equation}
with $f_{\rm even}({\color{black}x},{\color{black}y}, v) = (f({\color{black}x},{\color{black}y},v)+f(-{\color{black}x},-{\color{black}y},-v))/2$. 
Indeed, on the one hand, by symmetry of the system $(-{\color{black}X}^\delta_t, - {\color{black}Y}^\delta_t,-{\color{black}V}^\delta_t)$ has the same stationary distribution as $({\color{black}X}^\delta_t,{\color{black}Y}^\delta_t,{\color{black}V}^\delta_t)$. Therefore, if $f$ is odd, then ${\pi}(f)=0$ [this follows also from (\ref{eq:represnu}) since $w_0(\mathfrak{s}_-,f) = w_0(\mathfrak{s}_+,f(-\cdot)) =w_0(\mathfrak{s}_+, - f) =- w_0(\mathfrak{s}_+, f) $ when $f$ is odd]. On the other hand, denoting  $\eps_n={\rm sgn}({\color{black}X}^\delta_{\tau_n})$, 
the strong Markov property implies that the excursions $(\eps_n {\color{black}X}^\delta_{\tau_n+t} , \eps_n {\color{black}Y}^\delta_{\tau_n+t} ,
\eps_n {\color{black}V}^\delta_{\tau_n+t})_{t\in [0,\tau_{n+1}-\tau_n]}$ are independent and identically distributed with the distribution of $({\color{black}X}^\delta_{t} ,{\color{black}Y}^\delta_t, {\color{black}V}^\delta_t)_{t\in [0,{\tau_1}]}$ starting from $\mathfrak{s}_+$.
Therefore, if $f$ is even, then ${\pi}(f) = \frac{1}{\EE_{\mathfrak{s}_+} [{\tau_1}]} \EE_{\mathfrak{s}_+} \big[ \int_0^{{\tau_1}} f ({\color{black}X}^\delta_s,{\color{black}Y}^\delta_s, {\color{black}V}^\delta_s) \textup{d} s\big]$.

\subsection{Power spectral density}
\label{subsec:psd}%
The power spectral density { \color{black}(PSD)} of the process ${\color{black}V}_t^\delta$ can be defined as \cite{MC2012} 
\begin{equation}
\label{eq:defpsd}
S_v(\omega) = \lim \limits_{T \to \infty} \frac{1}{T} S_{v,T}(\omega)
\quad  
\mbox{where}
\quad
S_{v,T}(\omega) =  \EE_\pi\Big[ \Big| \int_0^T {\color{black}V}^\delta_t  \exp (-\bi \omega t ) \textup{d} t \Big|^2 \Big]. 
\end{equation} 
We have
\begin{equation}
\label{eq:expressSvomega}
S_v(\omega) = 2 {\rm Re} \big( {\pi}\big( v \hat{\phi}(\omega,{\color{black}x},{\color{black}y},v)\big)\big)  ,
\end{equation}
where $\hat{\phi}$ is the solution of 
\begin{align}
\label{eq:hatphi1}
( \bi \omega- {\cal L}) \hat{\phi} = v \: \mbox{in} \; E.
\end{align} 
Remark: For $\omega=0$ the solution is unique up to an additive constant which does not play any role in the evaluation of $S_v(\omega) $.

{\it Proof.}
{\color{black}
We have $S_{v,T}(\omega)
=\int_0^T \int_0^T \EE_\pi [ {\color{black}V}^\delta_t  {\color{black}V}^\delta_{t'} ] \cos \big(  \omega (t-t') \big) \textup{d} t  \textup{d} t'$. 
Using the stationarity $\EE_\pi [ {\color{black}V}^\delta_t  {\color{black}V}^\delta_{t'} ] = \EE_\pi [ {\color{black}V}^\delta_{|t-t'|}  {\color{black}V}^\delta_0 ]$, we find
 $S_{v,T}(\omega)
= 2 \int_0^T (T-t) \EE_\pi [ {\color{black}V}^\delta_t  {\color{black}V}^\delta_0 ] \cos  (  \omega t ) \textup{d} t$.
Using Lebesgue's dominated convergence theorem and the integrability of $t \mapsto \EE_\pi [ {\color{black}V}^\delta_t  {\color{black}V}^\delta_0 ]$ (see Appendix \ref{app:intpsi}), we get
$$
\lim_{T\to+\infty}\frac{1}{T} S_{v,T}(\omega)= S_v(\omega) = 2\int_0^{+\infty} \EE_{\pi} [ {\color{black}V}^\delta_0 {\color{black}V}^\delta_t] \cos(\omega t) \textup{d} t.
$$
We can write $
\EE_{\pi} [{\color{black}V}^\delta_0 {\color{black}V}^\delta_t ] =\EE_{\pi} [{\color{black}V}^\delta_0 \EE[ {\color{black}V}^\delta_t | {\color{black}X}^\delta_0 ,{\color{black}Y}^\delta_0, {\color{black}V}^\delta_0   ] ] =  {\pi}\big( v \phi(t,{\color{black}x},{\color{black}y},v) \big) $,}
where $\phi(t,{\color{black}x},{\color{black}y},v) = \EE[ {\color{black}V}^\delta_t |{\color{black}X}^\delta_0={\color{black}x},{\color{black}Y}^\delta_0={\color{black}y},{\color{black}V}^\delta_0=v]$ is the solution of 
$$
(\partial_t - {\cal L}) \phi  = 0 \: \mbox{in} \: E ,\: t >0,\quad\quad
\phi(t=0,v,{\color{black}y},{\color{black}x})   = v. 
$$
The function $\hat{\phi}(\omega,{\color{black}x},{\color{black}y},v) = \int_0^\infty \phi(t,{\color{black}x},{\color{black}y},v) \exp(-\bi \omega t) \textup{d} t$ satisfies (\ref{eq:hatphi1})
because \\
$\int_0^{+\infty} \partial_t \phi(t,{\color{black}x},{\color{black}y},v) \exp(-\bi \omega t) \textup{d} t = -\phi(t=0,{\color{black}x},{\color{black}y},v)+\bi \omega \hat{\phi}(\omega,{\color{black}x},{\color{black}y},v)$.
We then get (\ref{eq:expressSvomega}).
\qed

\subsection{Probability of sticking}
\label{subsec:probastick}%
We are interested in the probability of sticking, that is the empirical proportion of time spent in the sticking phase:
\begin{equation}
\label{pstick}
\lim_{T \to +\infty}\frac{1}{T} \int_0^T {\bf 1}_{  D^0 } ({\color{black}X}^\delta_t,{\color{black}Y}^\delta_t,{\color{black}V}^\delta_t)\textup{d} t ,
\end{equation}
where $D^0=\{ ({\color{black}x}_k,0,0), k=-k_{\mus},\ldots,k_{\mus}\}$.
By ergodicity 
 the limit (\ref{pstick}) exists almost surely, is deterministic and its value is independent of the starting point $({\color{black}X}^\delta_0,{\color{black}Y}^\delta_0,{\color{black}V}^\delta_0)$ and given by
\begin{equation}
P_{\rm stick} =  \frac{\EE_{\mathfrak{s}_+} \big[ \tau_1- \hat{\tau}_1 \big]}{\EE_{\mathfrak{s}_+} \big[ {\tau_1} \big]} =1-  \frac{\EE_{\mathfrak{s}_+} \big[ \hat{\tau}_1 \big]}{\EE_{\mathfrak{s}_+} \big[ {\tau_1} \big]} .
\end{equation}
This number can be evaluated as follows.
\begin{prop}
Let $\hat{{\color{black}R}}({\color{black}x},{\color{black}y},v)$, $\check{{\color{black}R}}({\color{black}x},{\color{black}y},v)$, $({\color{black}x},{\color{black}y},v)\in E$, and ${\color{black}R}({\color{black}x})$, ${\color{black}x} \in \{ {\color{black}x}_{-k_{\mus}-1},\ldots,{\color{black}x}_{k_{\mus}+1}\}$,
 be the solutions of 
\begin{align}
{\cal L} \hat{{\color{black}R}} &= -1 \: \mbox{ in } \: E\backslash D^0 ,
\quad \quad \hat{{\color{black}R}} = 0 \: \mbox{ in } \: D^0 , \\
{\cal L} \check{{\color{black}R}} &= -1 \: \mbox{ in } \: E\backslash \{({\color{black}x}_{\pm k_{\mus}},0,0)\},
\quad \quad \check{{\color{black}R}}({\color{black}x}_{\pm k_{\mus}},0,0) = 0, \\
Q^{\delta} {\color{black}R} &= -1  \: \mbox{ in } \: \{ {\color{black}x}_{-k_{\mus}},\ldots,{\color{black}x}_{k_{\mus}}\} ,
\quad \quad {\color{black}R} ( {\color{black}x}_{\pm (k_{\mu}+1)}) = 0.
\end{align}
Then we have $\EE_{\mathfrak{s}_+} \big[ \hat{\tau}_1 \big] = \hat{{\color{black}R}}(\mathfrak{s}_+)$,
$\EE_{\mathfrak{s}_+} \big[ {\tau_1} \big]=\check{{\color{black}R}}(\mathfrak{s}_+)+{\color{black}R}({\color{black}x}_{k_{\mus}})$,
 and
\begin{equation}
\label{pstick_formula_ratio}
P_{\rm stick}= 1 - \frac{\hat{{\color{black}R}}(\mathfrak{s}_+)}{\check{{\color{black}R}}(\mathfrak{s}_+)+{\color{black}R}({\color{black}x}_{k_{\mus}})} .
\end{equation}
\end{prop}

Note that we also have $\EE_{\mathfrak{s}_+} \big[ \hat{\tau}_1 \big]  = w_0(\mathfrak{s}_+ , {\bf 1}_{E\backslash D^0})$ and $\EE_{\mathfrak{s}_+} \big[ {\tau_1} \big]= w_0(\mathfrak{s}_+,1)=w_0(\mathfrak{s}_+ , {\bf 1}_{E\backslash D^0})+w_0(\mathfrak{s}_+ , {\bf 1}_{D^0})$, where $w_0$ has been introduced in 
(\ref{eq:defwlambda}), so that we can also write
\begin{equation}
P_{\rm stick}=1- \frac{w_0(\mathfrak{s}_+ , {\bf 1}_{E\backslash D^0})}{w_0(\mathfrak{s}_+ , {\bf 1}_{E\backslash D^0})+w_0(\mathfrak{s}_+ , {\bf 1}_{D^0})} 
=
\frac{w_0(\mathfrak{s}_+ , {\bf 1}_{D^0})}{w_0(\mathfrak{s}_+ , {\bf 1}_{E\backslash D^0})+w_0(\mathfrak{s}_+ , {\bf 1}_{D^0})} 
.
\end{equation}


\begin{proof}
We introduce
$$
\check{\tau}_1=
 \inf \big\{ t \geq 0,  \,  {\color{black}V}^\delta_t = 0 \mbox{ and } {\color{black}X}^\delta_t \in \{ -{\color{black}x}_{k_{\mus}},{\color{black}x}_{k_{\mus}} \} \big\} .
$$
We have 
$\EE_{({\color{black}x},{\color{black}y},v)} \big[ \hat{\tau}_1 \big] = \hat{{\color{black}R}}({\color{black}x},{\color{black}y},v )$ for any $({\color{black}x},{\color{black}y},v)\in E\backslash D^0$,
$\EE_{({\color{black}x},{\color{black}y},v)} \big[ \check{\tau}_1 \big] =\check{{\color{black}R}}({\color{black}x},{\color{black}y},v)$ for any $({\color{black}x},{\color{black}y},v)\in E$,
and
$\EE_{({\color{black}x},0,0)}[{\tau}_1] = {{\color{black}R}}({\color{black}x})$ for any ${\color{black}x}\in \{{\color{black}x}_{-k_{\mus}},\ldots,{\color{black}x}_{k_{\mus}}\}$.
Moreover, by the strong Markov property,
\begin{align*}
\EE_{\mathfrak{s}_+} \big[ {\tau}_1 \big] 
&=
\EE_{\mathfrak{s}_+} \big[ \check{\tau}_1 \big] + \EE_{\mathfrak{s}_+} \big[ \EE[ \tau_1-\check{\tau}_1 | \check{\tau}_1] \big] \\
&=
\EE_{\mathfrak{s}_+} \big[ \check{\tau}_1 \big] 
+ \EE_{({\color{black}x}_{k_{\mus}},0,0)}[ \tau_1 ] \PP_{\mathfrak{s}_+} \big( {\color{black}X}^\delta_{\check{\tau}_1} ={\color{black}x}_{k_{\mus}}\big) 
+ \EE_{({\color{black}x}_{-k_{\mus}},0,0)}[ \tau_1  ]\PP_{\mathfrak{s}_+} \big( {\color{black}X}^\delta_{\check{\tau}_1} ={\color{black}x}_{-k_{\mus}}\big)   .
\end{align*}
By symmetry of the system we have 
$\EE_{({\color{black}x}_{-k_{\mus}},0,0)}[\tau_1] = \EE_{({\color{black}x}_{k_{\mus}},0,0)}[\tau_1]$, so that
$\EE_{\mathfrak{s}_+} \big[ {\tau}_1 \big] 
=
\EE_{\mathfrak{s}_+} \big[ \check{\tau}_1 \big] 
+
\EE_{({\color{black}x}_{k_{\mus}},0,0)} \big[ {\tau}_1 \big] 
$ and we get
$$
\EE_{\mathfrak{s}_+} \big[ {\tau}_1 \big] =
\check{{\color{black}R}}(\mathfrak{s}_+)+{\color{black}R}({\color{black}x}_{k_{\mus}}),
$$
which completes the proof of the proposition.
\end{proof}

\subsection{Distributions of sticking and sliding periods}
\label{subsec:dist}%
The dynamics of the system consists of an alternate sequence of sticking periods and sliding (dynamic) periods.
Each sliding period starts from $\mathfrak{s}_\pm$ and the system is symmetric for the transform $({\color{black}x},{\color{black}y},v)\to(-{\color{black}x},-{\color{black}y},-v)$.
As a result the Laplace transform of the distribution of the duration of a sticking period is
\begin{equation}
F_{\rm stick}(\lambda )=\EE_{\mathfrak{s}_+} \big[ e^{- \lambda (\tau_1-\hat{\tau}_1)}\big]  .
\end{equation}
This Laplace transform can be evaluated as follows.
\begin{prop}
For $|k|\leq k_{\mus}$ and $\lambda>0$, 
let $\hat{P}_{k}({\color{black}x},{\color{black}y},v)$, $({\color{black}x},{\color{black}y},v)\in E$, and $F_{\lambda}({\color{black}x})$, ${\color{black}x} \in \{ {\color{black}x}_{-k_{\mus}-1},\ldots,{\color{black}x}_{k_{\mus}+1}\}$ be the solutions of 
\begin{align}
& {\cal L} \hat{P}_{k} = 0 \: \mbox{ in } \:E \backslash D^0 ,
\quad \quad \hat{P}_{k}({\color{black}x}_k,0,0)=1, \quad \quad \hat{P}_{k} =0\: \mbox{ in } \: D^0  \backslash \{({\color{black}x}_k,0,0)\}, \\
& \lambda F_{\lambda} - Q^\delta  F_{\lambda} = 0 \:\mbox{ in } \: \{ {\color{black}x}_{-k_{\mus}},\ldots,{\color{black}x}_{k_{\mus}}\}  \:  ,
\quad \quad  F_{\lambda}({\color{black}x}_{\pm (k_{\mus}+1)} )= 1.
\end{align}
Then we have $\PP_{\mathfrak{s}_+} \big( {\color{black}X}^\delta_{\hat{\tau}_1} ={\color{black}x}_k \big) = \hat{P}_{k}(\mathfrak{s}_+)$ for any $k\in \{-k_\mus,\ldots,k_\mus\}$,
$\EE_{({\color{black}x},0,0)} \big[ e^{-\lambda {\tau_1} } \big]= F_{\lambda}({\color{black}x})$ for any ${\color{black}x}\in \{ {\color{black}x}_{-k_{\mus}},\ldots,{\color{black}x}_{k_{\mus}}\} $ and $\lambda>0$,
 and
\begin{equation}
F_{\rm stick}(\lambda) = \sum_{k=-k_{\mus}}^{k_{\mus}} F_{\lambda}({\color{black}x}_k) \hat{P}_{k}(\mathfrak{s}_+)
.\end{equation}
\end{prop}

\begin{proof}
This is a consequence of the strong Markov property:
$$
F_{\rm stick}(\lambda) =
 \EE_{\mathfrak{s}_+} \big[ \EE[ e^{-\lambda (\tau_1-\hat{\tau}_1)} | \hat{\tau}_1 , {\color{black}X}^\delta_{\hat{\tau}_1} ] \big] 
=
\sum_{k=-k_{\mus}}^{k_{\mus}}  \PP_{\mathfrak{s}_+} \big( {\color{black}X}^\delta_{\hat{\tau}_1} ={\color{black}x}_k \big)  \EE_{({\color{black}x}_k,0,0)} \big[ e^{-\lambda \tau_1}  \big] .
$$
\end{proof}

Similarly the Laplace transform of the distribution of the duration of a sliding period is
\begin{equation}
F_{\rm slide}(\lambda )=\EE_{\mathfrak{s}_+} \big[ e^{- \lambda \hat{\tau}_1}\big]  ,
\end{equation}
and it can be expressed as follows.
\begin{prop}
\label{prop:laplacesliding}
For  $\lambda>0$, 
let $G_{\lambda}({\color{black}x},{\color{black}y},v)$, $({\color{black}x},{\color{black}y},v)\in E$ be the solution of
\begin{equation}
\lambda G_{\lambda} - {\cal L} G_{\lambda} = 0 \:\mbox{ in } \: E \backslash D^0   \:  ,
\quad \quad  G_{\lambda}= 1\:\mbox{ in } \: D^0   \:   .
\end{equation}
Then we have
\begin{equation}
F_{\rm slide}(\lambda )= G_\lambda(\mathfrak{s}_+).
\end{equation}
\end{prop}

From (\ref{eq:defwlambda}) we also have
\begin{align}
\nonumber
w_\lambda(\mathfrak{s}_+, {\bf 1}_{E\backslash D^0})
&= \mathbb{E}_{\mathfrak{s}_+} \Big[ \int_0^{\tau_1} e^{-\lambda s} {\bf 1}_{ E\backslash D^0} ({\color{black}X}^\delta_s,{\color{black}Y}^\delta_s,{\color{black}V}^\delta_s) \textup{d} s \Big]  
= \mathbb{E}_{\mathfrak{s}_+} \Big[ \int_0^{\hat{\tau}_1} e^{-\lambda s} \textup{d} s \Big]\\ 
& = \frac{1}{\lambda}
\big( 1- \EE_{\mathfrak{s}_+} \big[ e^{-\lambda \hat{\tau}_1} \big] \big),
\end{align}
so that we get $F_{\rm slide}(\lambda )=1-  \lambda w_\lambda(\mathfrak{s}_+, {\bf 1}_{E\backslash D^0})$.



\section{Numerics}
In this section, we are interested in the numerical computation of the following statistics under the stationary measure:
\begin{equation}
\label{eq:defSi}
\hspace*{-0.15in}
S^1= P_{{\rm stick}}, \quad \quad S^2 = \EE_\pi[ ({\color{black}V}^\delta_t)^2], \quad \quad S^3 = \mathbb{P}_\pi \left ( |{\color{black}X}^\delta_t| \leq \mus \right ).
\end{equation}
We use two different methods. 
The first method is probabilistic and relies on the representation formula \eqref{eq:rep_excursion} of $S^i$ in terms of the excursions $\{ ({\color{black}X}^\delta_s,{\color{black}Y}^\delta_s,{\color{black}V}^\delta_s), \: s \in [0,\tau_1] \}$ starting from $\mathfrak{s}_+$.
$S^1$, resp. $S^2$ and $S^3$, has the form \eqref{eq:rep_excursion} with $f({\color{black}x},v) = \mathbf{1}_{ \{ v= 0, |{\color{black}x}| \leq \mus \} }$, resp. $f({\color{black}x},v) = v^2$ and $f({\color{black}x},v) = \mathbf{1}_{ \{ |{\color{black}x}| \leq \mus \} }$. The second method is deterministic and consists in solving the equation \eqref{plambda} with $f$ in the right-hand side
and we look for $\lambda u_\lambda(f)$ for small $\lambda$.

\subsection{The probabilistic method: simulation of the excursions of $({\color{black}X}^\delta_s,{\color{black}Y}^\delta_s,{\color{black}V}^\delta_s)$ on $[0,\tau_1]$}
\label{subsec:approachMC}
We generate a large number, say ${\color{black}M}$, of independent and identically distributed (i.i.d.)  versions of the excursions $\{ ({\color{black}X}^\delta_s,{\color{black}Y}^\delta_s,{\color{black}V}^\delta_s), \: s \in [0,\tau_1] \}$ where $({\color{black}X}^\delta_0,{\color{black}Y}^\delta_0,{\color{black}V}^\delta_0) = \mathfrak{s}_+$. 
From this family of excursions, we construct ${\color{black}M}$ i.i.d. versions $\{{\color{black}\boxi}^{(i)} \}_{i=1}^M$ of  
\begin{equation}
{\color{black}\boxi =(\xi_1,\xi_2)^T} = \Big( \int_0^{\tau_1} f({\color{black}X}^\delta_s,{\color{black}V}^\delta_s) \textup{d} s , \tau_1\Big)^T,
\end{equation}
whose empirical mean and covariance are denoted by $\color{black}{\hat{\boxi}^{(M)}}$ and ${\color{black}\hat{\bf C}^{(M)}}$ respectively:
{\color{black}
\begin{equation}
{\color{black}\hat \boxi}^{(M)} = \frac{1}{M} \sum_{i=1}^M \boxi^{(i)},\quad
\hat{\bf C}^{(M)}
=\frac{1}{M}\sum_{i=1}^M \boxi^{(i)} (\boxi^{(i)})^T - \hat \boxi^{(M)} (\hat{\boxi}^{(M)})^T.
\end{equation}
}
From the central limit theorem, we have the convergence in distribution of ${\color{black}\hat \boxi}^{(M)}$:
{\color{black}
\begin{equation}
\sqrt{M} \big(\hat \boxi^{(M)} - \EE[\boxi] \big)
\stackrel{M\to +\infty}{\longrightarrow} 
{\cal N}\big( {\bf 0}_{\mathbb{R}^2}, {\bf C} \big), 
\end{equation}
where ${\bf C} = \EE[\boxi \boxi^T] - \EE[\boxi]\EE[\boxi]^T$.}
By the delta method,
we get 
\begin{equation}
\sqrt{{\color{black}M}}  
\Big(
 \frac{ {\color{black} \hat{\xi}_1^{(M)} }  }{
 {\color{black}\hat{\xi}_2^{(M)} } }
-
\frac{ \EE[{\color{black}\xi_1}] }{\EE[{\color{black}\xi_2}]}\Big)
\stackrel{{\color{black}M} \to +\infty}{\longrightarrow} 
{\cal N} (  0,\sigma^2),\quad \quad 
\sigma^2 = \begin{pmatrix}
1/\EE[{\color{black}\xi_2}]\\
-\EE[{\color{black}\xi_1}]/\EE[{\color{black}\xi_2}]^2
\end{pmatrix}^T {\bf C}
\begin{pmatrix}
1/\EE[{\color{black}\xi_2}]\\
-\EE[{\color{black}\xi_1}]/\EE[{\color{black}\xi_2}]^2
\end{pmatrix}.
\end{equation}
By Slutsky's theorem,
\begin{equation}
\frac{ \sqrt{{\color{black}M}} }{ \hat{\sigma}^{\color{black}(M)} }
\Big(
 \frac{ {\color{black} \hat{\xi}_1^{(M)} }  }{
 {\color{black}\hat{\xi}_2^{(M)} } }
-
\frac{ \EE[{\color{black} {\xi}_1 }] }{\EE[{\color{black}  {\xi}_2 }]}\Big)
\stackrel{{\color{black}M} \to +\infty}{\longrightarrow} 
{\cal N} (  0,1),\quad \quad 
(\hat{\sigma}^{{\color{black}M}})^2 = \begin{pmatrix}
1/{\color{black} \hat{\xi}_2^{(M)} }\\
-{\color{black} \hat{\xi}_1^{(M)} } / ({\color{black} \hat{\xi}_2^{(M)} })^2
\end{pmatrix}^T \hat{\bf C}_{\color{black}M}
\begin{pmatrix}
1/{\color{black} \hat{\xi}_2^{(M)} }\\
- {\color{black} \hat{\xi}_1^{(M)} }/ ({\color{black} \hat{\xi}_2^{(M)} })^2
\end{pmatrix}  .
\end{equation}
We can deduce from this convergence in distribution an asymptotic 95 \% confidence interval for $\mathbb{E} [{\color{black}\xi_1}]/\EE[{\color{black}\xi_2}]$ (which is the quantity of interest by (\ref{eq:rep_excursion})):
\begin{equation}
\PP\left(
\frac{\mathbb{E} [{\color{black}\xi_1}]}{\mathbb{E} [{\color{black}\xi_2}] } \in \Big( 
\frac{\color{black} \hat{\xi}_1^{(M)} }{\color{black} \hat{\xi}_2^{(M)} }  - 1.96 \hat{\sigma}^{\color{black}(M)} {\color{black}M}^{-\frac{1}{2}},
\frac{\color{black} \hat{\xi}_1^{(M)} }{\color{black} \hat{\xi}_2^{(M)} } + 1.96 \hat{\sigma}^{\color{black}(M)} {\color{black}M}^{-\frac{1}{2}}
\Big)
 \right)
 \stackrel{{\color{black}M} \to +\infty}{\longrightarrow} 0.95. 
\end{equation}

Note that
\begin{align}
\nonumber
{\color{black}\xi_1} =  & \sum \limits_{0 \leq i \leq i_{\tau_1} -1} \int_{T_i}^{T_{i+1}} f({\color{black}X}^\delta_{T_i},\Phi_{{\color{black}X}^\delta_{T_i},{\color{black}Y}^\delta_{T_i}}(s-T_i , {\color{black}V}^\delta_{T_i})) \textup{d} s\\ 
& + \int_{T_{i_{\tau_1}}}^{\tau_1} f({\color{black}X}^\delta_{T_{i_{\tau_1}}},\Phi_{{\color{black}X}^\delta_{T_{i_{\tau_1}}},{\color{black}Y}^\delta_{T_{i_{\tau_1}}}}( s-T_{i_{\tau_1}} , {\color{black}V}^\delta_{T_{i_{\tau_1}}})) \textup{d} s ,
\end{align}
where $i_{\tau_1} = \max \{ i, \: T_i \leq \tau_1 \}.$
If $f$ does not depend on $v$ then the formula above becomes simple
\begin{align}
\int_0^{\tau_1} f({\color{black}X}^\delta_s) \textup{d} s  
= \sum \limits_{0 \leq i \leq i_{\tau_1}-1} f({\color{black}X}^\delta_{T_i}) (T_{i+1}-T_i) 
+ f({\color{black}X}^\delta_{T_{i_{\tau_1}}}) (\tau_1 - T_{i_{\tau_1}}).
\end{align}
With the particular choice of $f({\color{black}x},v) = \mathbf{1}_{\left \{ v= 0, \: | {\color{black}x} | \leq \mus \right \}}$, the formula remains simple as well
\begin{align}
\int_0^{\tau_1} f({\color{black}X}^\delta_s,{\color{black}V}^\delta_s) \textup{d} s  
= \sum \limits_{0 \leq i \leq i_{\tau_1}-1} f({\color{black}X}^\delta_{T_i},{\color{black}V}^\delta_{T_i}) (T_{i+1}-T_i) 
+ f({\color{black}X}^\delta_{T_{i_{\tau_1}}},{\color{black}V}^\delta_{T_{i_{\tau_1}}}) (\tau_1 - T_{i_{\tau_1}}).
\end{align}
When $f({\color{black}x},v) = v^2$ then it becomes slightly more complicated
\begin{align}
\nonumber
\int_0^{\tau_1} f({\color{black}X}^\delta_s,{\color{black}V}^\delta_s) \textup{d} s  
= &  \sum \limits_{0 \leq i \leq i_{\tau_1}-1} \int_{T_i}^{T_{i+1}} \big( \Phi_{{\color{black}X}^\delta_{T_i},{\color{black}Y}^\delta_{T_i}}(s-T_i , {\color{black}V}^\delta_{T_i}) \big)^2 \textup{d} s \mathbf{1}_{\{ |{\color{black}X}^\delta_{T_i}|=1\}}\\
& +\int_{T_{i_{\tau_1}}}^{\tau_1} \big( \Phi_{{\color{black}X}^\delta_{T_{i_{\tau_1}}},{\color{black}Y}^\delta_{T_{i_{\tau_1}}}}( s-T_{i_{\tau_1}} , {\color{black}V}^\delta_{T_{i_{\tau_1}}}) \big)^2 \textup{d} s \mathbf{1}_{\{ |{\color{black}Y}^\delta_{T_{i_{\tau_1}}}|=1\}}.
\end{align}
If { \color{black} $\mathfrak{b}(v)=-v$}, then each of the integrals in the right-hand side can be computed explicitly:
\begin{align}
\nonumber
\int_0^{\tau_1} f({\color{black}X}^\delta_s,{\color{black}V}^\delta_s) \textup{d} s  = &
\sum \limits_{0 \leq i \leq i_{\tau_1}-1} \Psi(T_i,T_{i+1};{\color{black}X}^\delta_{T_i},{\color{black}Y}^\delta_{T_i},{\color{black}V}^\delta_{T_i}) \mathbf{1}_{\{ |{\color{black}Y}^\delta_{T_i}|=1\}}
\\
& + \Psi(T_{i_{\tau_1}},\tau_1;{\color{black}X}^\delta_{T_{i_{\tau_1}}},{\color{black}Y}^\delta_{T_{i_{\tau_1}}},{\color{black}V}^\delta_{T_{i_{\tau_1}}}) \mathbf{1}_{\{ |{\color{black}X}^\delta_{T_{i_{\tau_1}}}|=1\}}  ,
\end{align}
with
$\Psi(\theta,\theta';{\color{black}x},{\color{black}y},v) = (\theta'-\theta) ({\color{black}x}-{\color{black}y} \mud)^2 + 0.5 (v-{\color{black}x}+{\color{black}y} \mud)^2 (1-e^{-2(\theta'-\theta)}) 
+ 2 ({\color{black}x}-{\color{black}y} \mud)(v-{\color{black}x}+{\color{black}y} \mud) (1-e^{-(\theta'-\theta)}).$
In Figure \ref{fig:probabilisticresults}, an estimation of the three quantities $S^i$ as functions of $\delta$ is presented together with the error bars using the probabilistic method and the formulas above. 
\begin{figure}[h!]
\centering
\begin{tikzpicture}[scale=0.49]
\begin{axis}[legend style={at={(0,1)},anchor=north west}, compat=1.3,
  xmin=0.03, xmax=.5,ymin=0.10,ymax=0.6,xmode = log, log basis x = {2},
  xlabel= {$\delta$},
  ylabel= {$P_\textup{stick}$}]
\addplot[dashdotted,mark=triangle*,mark size=2pt,mark options={color=black}] table [x index=0, y index=3]{casea_pdmp_friction_data6.txt};
\addplot[solid,mark size=1pt,mark options={color=black}] table [x index=0, y index=4]{casea_pdmp_friction_data6.txt};
\addplot[solid,mark size=1pt,mark options={color=black}] table [x index=0, y index=5]{casea_pdmp_friction_data6.txt};
\addplot[dashdotted,mark=triangle*,mark size=2pt,mark options={color=black}] table [x index=0, y index=3]{caseb_pdmp_friction_data6.txt};
\addplot[solid,mark size=1pt,mark options={color=black}] table [x index=0, y index=4]{caseb_pdmp_friction_data6.txt};
\addplot[solid,mark size=1pt,mark options={color=black}] table [x index=0, y index=5]{caseb_pdmp_friction_data6.txt};
\addplot[dashdotted,mark=triangle*,mark size=2pt,mark options={color=black}] table [x index=0, y index=3]{casec_pdmp_friction_data6.txt};
\addplot[solid,mark size=1pt,mark options={color=black}] table [x index=0, y index=4]{casec_pdmp_friction_data6.txt};
\addplot[solid,mark size=1pt,mark options={color=black}] table [x index=0, y index=5]{casec_pdmp_friction_data6.txt};
\addplot[dashdotted,mark=triangle*,mark size=2pt,mark options={color=black}] table [x index=0, y index=3]{cased_pdmp_friction_data6.txt};
\addplot[solid,mark size=1pt,mark options={color=black}] table [x index=0, y index=4]{cased_pdmp_friction_data6.txt};
\addplot[solid,mark size=1pt,mark options={color=black}] table [x index=0, y index=5]{cased_pdmp_friction_data6.txt};
\end{axis}
\end{tikzpicture}
\begin{tikzpicture}[scale=0.49]
\begin{axis}[legend style={at={(0,1)},anchor=north west}, compat=1.3,
  xmin=0.03, xmax=.5,ymin=0.03,ymax=0.352, xmode = log, log basis x = {2},
  xlabel= {$\delta$},
  ylabel= {$\mathbb{E}[{\color{black}V^\delta}^2]$}]
\addplot[dashdotted,mark=triangle*,mark size=2pt,mark options={color=black}] table [x index=0, y index=12]{casea_pdmp_friction_data6.txt};
\addplot[solid,mark size=1pt,mark options={color=black}] table [x index=0, y index=13]{casea_pdmp_friction_data6.txt};
\addplot[solid,mark size=1pt,mark options={color=black}] table [x index=0, y index=14]{casea_pdmp_friction_data6.txt};
\addplot[dashdotted,mark=triangle*,mark size=2pt,mark options={color=black}] table [x index=0, y index=12]{caseb_pdmp_friction_data6.txt};
\addplot[solid,mark size=1pt,mark options={color=black}] table [x index=0, y index=13]{caseb_pdmp_friction_data6.txt};
\addplot[solid,mark size=1pt,mark options={color=black}] table [x index=0, y index=14]{caseb_pdmp_friction_data6.txt};
\addplot[dashdotted,mark=triangle*,mark size=2pt,mark options={color=black}] table [x index=0, y index=12]{casec_pdmp_friction_data6.txt};
\addplot[solid,mark size=1pt,mark options={color=black}] table [x index=0, y index=13]{casec_pdmp_friction_data6.txt};
\addplot[solid,mark size=1pt,mark options={color=black}] table [x index=0, y index=14]{casec_pdmp_friction_data6.txt};
\addplot[dashdotted,mark=triangle*,mark size=2pt,mark options={color=black}] table [x index=0, y index=12]{cased_pdmp_friction_data6.txt};
\addplot[solid,mark size=1pt,mark options={color=black}] table [x index=0, y index=13]{cased_pdmp_friction_data6.txt};
\addplot[solid,mark size=1pt,mark options={color=black}] table [x index=0, y index=14]{cased_pdmp_friction_data6.txt};
\end{axis}
\end{tikzpicture}
\begin{tikzpicture}[scale=0.49]
\begin{axis}[legend style={at={(0,1)},anchor=north west}, compat=1.3,
  xmin=0.03125, xmax=0.5,ymin=0.686,ymax=0.793, xmode = log, log basis x = {2},
  xlabel= {$\delta$},
  ylabel= {$\mathbb{P} \left ( |{\color{black}X^\delta}^2| \leq \mus \right )$}]
\addplot[dashdotted,mark=triangle*,mark size=2pt,mark options={color=black}] table [x index=0, y index=6]{casea_pdmp_friction_data6.txt};
\addplot[solid,mark size=1pt,mark options={color=black}] table [x index=0, y index=7]{casea_pdmp_friction_data6.txt};
\addplot[solid,mark size=1pt,mark options={color=black}] table [x index=0, y index=8]{casea_pdmp_friction_data6.txt};
\addplot[color=red, thick, domain=0.4375:0.5625]{0.789886474609376};1
\addplot[color=red, thick, domain=0.21875:0.28125]{0.739564523423343};2
\addplot[color=red, thick, domain=0.109375:0.140625]{0.712010713468334};3
\addplot[color=red, thick, domain=0.0546875:0.0703125]{0.697578849841333};4
\addplot[color=red, thick, domain=0.02734375:0.03515625]{0.690192311713429};5

\end{axis}
\end{tikzpicture}
\caption{Monte Carlo estimations of $P_{\rm stick},\mathbb{E}[ {\color{black}V^\delta}^2], \mathbb{P}(|{\color{black}X^\delta}|\leq \mus)$ versus $\delta$ (taking values $2^{-j}$, $j=1,\ldots,5$) based on $N = 10^6$ simulated excursions of $({\color{black}X}^\delta,{\color{black}Y}^\delta,{\color{black}V}^\delta)$. 
Here $\mus=1$ { \color{black} and $\mathfrak{b}(v) = - v$}.   In the subfigure for $P_{\rm stick}$, the four curves from bottom to top correspond to $\mud = \frac{1}{4}$, $\frac{1}{2}$, $\frac{3}{4}$, and $1$ respectively whereas in the subfigure for $\mathbb{E}[ {\color{black}V^\delta}^2]$ the order is reversed. In the subfigure for $\mathbb{P}(|{\color{black}X^\delta} | \leq \mus)$, the red segments represent the numerical results associated with finding the invariant measure of ${\color{black}X^\delta}$ in the left kernel of $Q^\delta$ and then computing the targeted statistics. The asymptotic 95 \% confidence intervals are  represented by two solid lines around the estimated results.
\label{fig:probabilisticresults}}
\end{figure}

\subsection{The deterministic method: discretization in the $v$-axis of the $\lambda$-problem}
To numerically approximate the solution of \eqref{plambda}, we use a finite difference scheme where only the $v$-axis is discretized.
We consider a two-dimensional grid, for any $p \in \mathbb{N}^\star$, with $I = 2N +1$ and $J = 2 N p +1$,
\begin{equation}
\mathcal{G}_p = \left \{ ({\color{black}x}_i,v_j) = \Big( (i-N-1) \delta, (j-N p-1) \frac{\delta}{p} \Big), 
\: 1 \leq i \leq I,
\: 1 \leq j \leq J  
\right \} .
\end{equation}
The number of points in the grid $\mathcal{G}_p$ is $N_p = (2N+1)(2N p + 1) \sim 4N^2 p$ as $N \to \infty$.
The numerical approximation of $u_\lambda({\color{black}x}_i,\Theta({\color{black}x}_i,v_j),v_j)$ is denoted by $u_{ij}$ and the corresponding vector collecting the unknowns is $\boldsymbol{u}$. We also use the notation $f_{ij}$ for $f({\color{black}x}_i,\Theta({\color{black}x}_i,v_j),v_j)$ and $\boldsymbol{f}$ for the corresponding vector. 
We use a standard finite difference scheme in the $v$ direction: {\color{black} when $j \neq N p +1 (v_j \neq 0)$ or $|i - N - 1| > k_{\mus}$,}
\begin{equation}
\label{eq:discreteproblem1}
\lambda u_{ij} - ({\bf K} \boldsymbol{u})_{ij} - ( {\bf J} \boldsymbol{u})_{ij} =  {f}_{ij} 
\end{equation}
otherwise when  { \color{black}$j = N p +1$ ($v_j = 0$) and $|i - N - 1| \leq k_{\mus}$}
\begin{equation}
\label{eq:discreteproblem2}
\lambda u_{ij} - ({\bf J} \boldsymbol{u})_{ij} =  {f}_{ij},
\end{equation}
with 
$
({\bf J} \boldsymbol{u})_{ij} = 2\tau^{-2} \delta^{-2} \left ( \alpha_i u_{i+1j} - u_{ij} + (1-\alpha_i) u_{i-1j} \right )
$, $\alpha_i = \alpha({\color{black}x}_i)$, 
and 
\begin{equation}
({\bf K} \boldsymbol{u})_{ij} = p \max (0,B_{ij}) \left ( \frac{u_{ij+1}-u_{ij}}{\delta} \right ) + p \min (0,B_{ij}) \left ( \frac{u_{ij}-u_{ij-1}}{\delta} \right ),
\end{equation}
with $B_{ij}=B(x_i,{\color{black}\Theta(x_i,v_j)},v_j)$.
This results in a linear system to be solved of the form $(\lambda {\bf I} - {\bf M}) \boldsymbol{u} = \boldsymbol{f}$
where both ${\bf I}$ and ${\bf M}$ are $N_p \times N_p$ sparse matrices, ${\bf I}$ is the identity matrix and ${\bf M}$ is a sparse matrix with at most five nonzero entries per row.
The computational time spent to find $\boldsymbol{u}$ corresponds essentially to the $LU$ factorization of the matrix $\lambda {\bf I}-{\bf M}$ associated with the system \eqref{eq:discreteproblem1}-\eqref{eq:discreteproblem2}. We employ the MATLAB procedure $lu(.)$ which seeks five invertible matrices ${\bf L},{\bf U},{\bf P},{\bf Q},{\bf D}$ where ${\bf L}$ aud ${\bf U}$ are resp. lower and upper triangular such that $\lambda {\bf I}-{\bf M}={\bf D}{\bf P}^{-1}{\bf L}{\bf U}{\bf Q}^{-1}.$ 
As shown in Figure~\ref{fig4}, for $p$ large enough we recover the results of the probabilistic approach of Subsection~\ref{subsec:approachMC}.

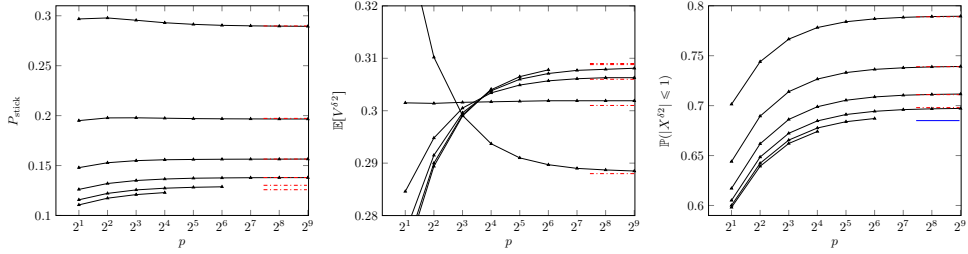
\begin{figure}[h!]
\centering
\begin{tikzpicture}[scale=0.49]
\begin{axis}[legend style={at={(0,1)},anchor=north west}, compat=1.3,
  xmode = log, log basis x = {2}, xmin=0, xmax=520,ymin=0.1,ymax=0.31,
  xlabel= {$p$},
  ylabel= {$P_{\rm stick}$}]
\addplot[color=black,mark=triangle,mark size=1pt] table [x index=0, y index=4]{casea_lambdaproblem_delta0.5_data.txt};
\addplot[color=red, thick, dashdotted, domain=175:500]{0.2898};
\addplot[color=black,mark=triangle,mark size=1pt] table [x index=0, y index=4]{casea_lambdaproblem_delta0.25_data.txt};
\addplot[color=red, thick, dashdotted, domain=175:500]{0.1972};
\addplot[color=black,mark=triangle,mark size=1pt] table [x index=0, y index=4]{casea_lambdaproblem_delta0.125_data.txt};
\addplot[color=red, thick, dashdotted, domain=175:500]{0.1565};
\addplot[color=black,mark=triangle,mark size=1pt] table [x index=0, y index=4]{casea_lambdaproblem_delta0.0625_data.txt};
\addplot[color=red, thick, dashdotted, domain=175:500]{0.1380};
\addplot[color=black,mark=triangle,mark size=1pt] table [x index=0, y index=4]{casea_lambdaproblem_delta0.03125_data.txt};
\addplot[color=red, thick, dashdotted, domain=175:500]{0.1302};
\addplot[color=black,mark=triangle,mark size=1pt] table [x index=0, y index=4]{casea_lambdaproblem_delta0.015625_data.txt};
\addplot[color=red, thick, dashdotted, domain=175:500]{0.1258};
\end{axis}
\end{tikzpicture}
\begin{tikzpicture}[scale=0.49]
\begin{axis}[legend style={at={(0,1)},anchor=north west}, compat=1.3,
  xmode = log, log basis x = {2}, xmin=0, xmax=520,ymin=0.28,ymax=0.32,
  xlabel= {$p$},
  ylabel= {$\mathbb{E}[{\color{black}V^\delta}^2]$}]
\addplot[color=black,mark=triangle,mark size=1pt] table [x index=0, y index=7]{casea_lambdaproblem_delta0.5_data.txt}; 
\addplot[color=red, thick, dashdotted, domain=175:2048]{0.288};
\addplot[color=black,mark=triangle,mark size=1pt] table [x index=0, y index=7]{casea_lambdaproblem_delta0.25_data.txt}; 
\addplot[color=red, thick, dashdotted, domain=175:2048]{0.301};
\addplot[color=black,mark=triangle,mark size=1pt] table [x index=0, y index=7]{casea_lambdaproblem_delta0.125_data.txt}; 
\addplot[color=red, thick, dashdotted, domain=175:2048]{0.306};
\addplot[color=black,mark=triangle,mark size=1pt] table [x index=0, y index=7]{casea_lambdaproblem_delta0.0625_data.txt}; 
\addplot[color=red, thick, dashdotted, domain=175:2048]{0.309};
\addplot[color=black,mark=triangle,mark size=1pt] table [x index=0, y index=7]{casea_lambdaproblem_delta0.03125_data.txt}; 
\addplot[color=red, thick, dashdotted, domain=175:2048]{0.3088};
\addplot[color=black,mark=triangle,mark size=1pt] table [x index=0, y index=7]{casea_lambdaproblem_delta0.015625_data.txt}; 
\end{axis}
\end{tikzpicture}
\begin{tikzpicture}[scale=0.49]
\begin{axis}[legend style={at={(0,1)},anchor=north west}, compat=1.3,
  xmode = log, log basis x = {2}, xmin=0, xmax=520,ymin=0.59,ymax=0.8,
  xlabel= {$p$},
  ylabel= {$\mathbb{P}(|{\color{black}X^\delta}^2| \leq 1)$}]
\addplot[color=black,mark=triangle,mark size=1pt] table [x index=0, y index=5]{casea_lambdaproblem_delta0.5_data.txt}; 
\addplot[color=red, thick, dashdotted, domain=175:500]{0.789};
\addplot[color=black,mark=triangle,mark size=1pt] table [x index=0, y index=5]{casea_lambdaproblem_delta0.25_data.txt}; 
\addplot[color=red, thick, dashdotted, domain=175:500]{0.739};
\addplot[color=black,mark=triangle,mark size=1pt] table [x index=0, y index=5]{casea_lambdaproblem_delta0.125_data.txt}; 
\addplot[color=red, thick, dashdotted, domain=175:500]{0.711};
\addplot[color=black,mark=triangle,mark size=1pt] table [x index=0, y index=5]{casea_lambdaproblem_delta0.0625_data.txt}; 
\addplot[color=red, thick, dashdotted, domain=175:500]{0.698};
\addplot[color=black,mark=triangle,mark size=1pt] table [x index=0, y index=5]{casea_lambdaproblem_delta0.03125_data.txt}; 
\addplot[color=black,mark=triangle,mark size=1pt] table [x index=0, y index=5]{casea_lambdaproblem_delta0.015625_data.txt}; 
\addplot[color=blue, thick, domain=175:500]{0.685};
\end{axis}
\end{tikzpicture}
\caption{
Case $\mud=0.25\mus$ {\color{black} and $\mathfrak{b}(v) = -v$}.
Approximations of $P_{\rm stick}, \mathbb{E}[ {\color{black}V^\delta}^2] , \mathbb{P}(|{\color{black}X^\delta}|\leq 1)$ versus $2 \leq p \leq 2^9$ 
for several values of $\delta$ (from $2^{-1}$ halved successively until $2^{-6}$). The $p$-axis is represented in the scale of $\log$ base 2.
In red (the limit case in $p$), the Monte Carlo result (plotted in Figure \ref{fig:probabilisticresults}) is shown for comparison. In blue (the limit case in $p$), the theoretical value is plotted when available.
\label{fig4}}
\end{figure}

\paragraph{Empirical convergence rate w.r.t $p$}
If the finite difference scheme \eqref{eq:discreteproblem1}-\eqref{eq:discreteproblem2} is of order $\kappa$ then $\| u^p - u \| \leq Cp^{-\kappa}$ where $C$ is independent of $p$. Moreover, if there exists an $\epsilon >0$ such that $\| u^p - u \| = Cp^{-\kappa} + O(p^{-\kappa-\epsilon})$
then 
$
\| u^{2p} - u^p \| \| u^p - u^{\frac{p}{2}} \|^{-1} \approx 2^\kappa + O(p^{-\epsilon}).
$
With such a relation in mind, we test the convergence of the finite difference scheme by considering
$
\kappa(p) = \log_2 \left ( \| u^{2p} - u^p \| \| u^p - u^{\frac{p}{2}} \|^{-1} \right ).
$
In Table \ref{tab:kp}, we present a set of empirical estimations of $\kappa(p)$ in two cases. 
The data indicate that $\kappa(p) \sim 1$.
\begin{table}[h!]
\centering
{\scriptsize
\begin{tabular}{||c c c c c||} 
 \hline
$\delta = 2^{-1}$ & $\kappa(64)$ & $\kappa(128)$ & $\kappa(256)$ & $\kappa(512)$ \\ 
 \hline\hline
 $S^{1}$ & 0.881 & 0.939  & 0.976  & 0.997 \\
  $S^{2}$ & 0.997 & 0.999  & 0.996  & 0.986 \\
 $S^{3}$ & 0.997 & 0.999  & 1.002  & 1.010 \\
  \hline
\end{tabular}
\begin{tabular}{||c c c c c||} 
 \hline
$\delta = 2^{-2}$ & $\kappa(64)$ & $\kappa(128)$ & $\kappa(256)$ & $\kappa(512)$ \\ 
 \hline\hline
 $S^{1}$ & 0.865 & 0.941  & 0.995  & 0.988 \\
 $S^{2}$ & 0.995 & 0.996  & 0.992  & 0.998 \\
 $S^{3}$ & 0.994 & 1.000  & 1.010  & 1.001 \\
  \hline
\end{tabular}
}
\caption{Computation of $\kappa(p)$ when $\delta = 2^{-1}$ on the left, $\delta = 2^{-2}$ on the right. The method is empirically of order 1. { \color{black} Here $\mathfrak{b}(v) = -v$}}
\label{tab:kp}
\end{table}

\subsection{Discussion}
\paragraph{Power spectral density}
{
\color{black}We discuss the $\tau$ dependence of the PSD of the velocity $V^\delta$  and the corresponding linewidth (from which we obtain the correlation time). We rely on numerical investigation since explicit expressions are not known.
We calculate the PSD of the velocity $V^\delta$ by solving the equations \eqref{eq:expressSvomega} and \eqref{eq:hatphi1}.
The Monte Carlo method is used for verification (with $T=10^3$ and $10^4$ sample paths). Figure \ref{figX} shows numerical results of the PSD for different values of $\mud, \delta$ and $\tau$. For each value of $\mud$, when $\tau$ goes below $0.25$ the curves with three different values of $\delta$ become indistinguishable. This indicates that for $\delta$ small enough ($\leq 0.25$), the main driving parameter becomes $\tau$.}
\begin{figure}[h!]
\centering
\begin{tikzpicture}[scale=0.49]
\begin{loglogaxis}[legend style={at={(0,0)},anchor=south west}, compat=1.3,
  xmin=0.001, xmax=50,ymin=0.,ymax=1.5,
  xlabel= {$\omega$},
  ylabel= {$S_V(\omega)$}]
\addplot[thick, color=gray, densely dotted,mark size=1pt] table [x index=0, y index=1]{mud025/probab_psd_delta025_tau1_mud0.25.txt};
\addplot[forget plot, each nth point=5, only marks,mark=*,mark size=1pt] table [x index=0, y index=1]{mud025/determ_psd_delta025_tau1_mud0.25.txt};
\addlegendentry{$\tau=1$};
\addplot[thick, color=gray,dashdotted,mark=none,mark size=1pt] table [x index=0, y index=1]{mud025/probab_psd_delta025_tau05_mud0.25.txt};
\addplot[forget plot, each nth point=5, only marks,mark=*,mark size=1pt] table [x index=0, y index=1]{mud025/determ_psd_delta025_tau05_mud0.25.txt};
\addlegendentry{$\tau=2^{-1}$};
\addplot[thick,color=gray,mark size=1pt] table [x index=0, y index=1]{mud025/probab_psd_delta025_tau025_mud0.25.txt};
\addplot[forget plot, each nth point=5, only marks,mark=*,mark size=1pt] table [x index=0, y index=1]{mud025/determ_psd_delta025_tau025_mud0.25.txt};
\addlegendentry{$\tau=2^{-2}$};
 \node at (-1,-9) {\textcolor{black}{$\mud=0.25, \delta = 1$}};
\end{loglogaxis}
\end{tikzpicture}
\begin{tikzpicture}[scale=0.49]
\begin{loglogaxis}[legend style={at={(0,0)},anchor=south west}, compat=1.3,
  xmin=0.001, xmax=50,ymin=0.,ymax=1.5,
  xlabel= {$\omega$},
  ylabel= {$S_V(\omega)$}]
 \addplot[thick, color=gray,densely dotted,mark=none,mark size=1pt] table [x index=0, y index=1]{mud05/probab_psd_delta025_tau1_mud0.5.txt};
\addplot[forget plot, each nth point=5, only marks,mark=*,mark size=1pt] table [x index=0, y index=1]{mud05/determ_psd_delta025_tau1_mud0.5.txt};
\addlegendentry{$\tau=1$};
\addplot[thick, color=gray,dashdotted,mark=none,mark size=1pt] table [x index=0, y index=1]{mud05/probab_psd_delta025_tau05_mud0.5.txt};
\addplot[forget plot, each nth point=5, only marks,mark=*,mark size=1pt] table [x index=0, y index=1]{mud05/determ_psd_delta025_tau05_mud0.5.txt};
\addlegendentry{$\tau=2^{-1}$};
\addplot[thick,color=gray,mark=none,mark size=1pt] table [x index=0, y index=1]{mud05/probab_psd_delta025_tau025_mud0.5.txt};
\addplot[forget plot, each nth point=5, only marks,mark=*,mark size=1pt] table [x index=0, y index=1]{mud05/determ_psd_delta025_tau025_mud0.5.txt};
\addlegendentry{$\tau=2^{-2}$};
\node at (-1,-9) {\textcolor{black}{$\mud=0.5, \delta = 1$}};
\end{loglogaxis}
\end{tikzpicture}
\begin{tikzpicture}[scale=0.49]
\begin{loglogaxis}[legend style={at={(0,0)},anchor=south west}, compat=1.3,
  xmin=0.001, xmax=50,ymin=0.,ymax=1.5,
  xlabel= {$\omega$},
  ylabel= {$S_V(\omega)$}]
\addplot[thick, color=gray,densely dotted,mark=none,mark size=1pt] table [x index=0, y index=1]{mud1/probab_psd_delta025_tau1_mud1.txt};
\addplot[forget plot, each nth point=5, only marks,mark=*,mark size=1pt] table [x index=0, y index=1]{mud1/determ_psd_delta025_tau1_mud1.txt};
\addlegendentry{$\tau=1$};
\addplot[thick, color=gray,dashdotted,mark=none,mark size=1pt] table [x index=0, y index=1]{mud1/probab_psd_delta025_tau05_mud1.txt};
\addplot[forget plot, each nth point=5, only marks,mark=*,mark size=1pt] table [x index=0, y index=1]{mud1/determ_psd_delta025_tau05_mud1.txt};
\addlegendentry{$\tau=2^{-1}$};
\addplot[thick,color=gray,mark=none,mark size=1pt] table [x index=0, y index=1]{mud1/probab_psd_delta025_tau025_mud1.txt};
\addplot[forget plot, each nth point=5, only marks,mark=*,mark size=1pt] table [x index=0, y index=1]{mud1/determ_psd_delta025_tau025_mud1.txt};
\addlegendentry{$\tau=2^{-2}$};
\node at (-1,-9) {\textcolor{black}{$\mud=1, \delta = 1$}};
\end{loglogaxis}
\end{tikzpicture}

\begin{tikzpicture}[scale=0.49]
\begin{loglogaxis}[legend style={at={(0,0)},anchor=south west}, compat=1.3,
  xmin=0.001, xmax=50,ymin=0.,ymax=1.5,
  xlabel= {$\omega$},
  ylabel= {$S_V(\omega)$}]
  \addplot[thick, color=gray,densely dotted,mark size=1pt] table [x index=0, y index=1]{mud025/probab_psd_delta05_tau1_mud0.25.txt};
\addplot[forget plot, each nth point=5, only marks,mark=*,mark size=1pt] table [x index=0, y index=1]{mud025/determ_psd_delta05_tau1_mud0.25.txt};
\addlegendentry{$\tau=1$};
\addplot[thick, color=gray,dashdotted, mark=none,mark size=1pt] table [x index=0, y index=1]{mud025/probab_psd_delta05_tau05_mud0.25.txt};
\addplot[forget plot, each nth point=5, only marks,mark=*,mark size=1pt] table [x index=0, y index=1]{mud025/determ_psd_delta05_tau05_mud0.25.txt};
\addlegendentry{$\tau=2^{-1}$};
\addplot[thick,color=gray,mark size=1pt] table [x index=0, y index=1]{mud025/probab_psd_delta05_tau025_mud0.25.txt};
\addplot[forget plot, each nth point=5, only marks,mark=*,mark size=1pt] table [x index=0, y index=1]{mud025/determ_psd_delta05_tau025_mud0.25.txt};
\addlegendentry{$\tau=2^{-2}$};
\node at (-1,-9) {\textcolor{black}{$\mud=0.25, \delta = 2^{-1}$}};
\end{loglogaxis}
\end{tikzpicture}
\begin{tikzpicture}[scale=0.49]
\begin{loglogaxis}[legend style={at={(0,0)},anchor=south west}, compat=1.3,
  xmin=0.001, xmax=50,ymin=0.,ymax=1.5,
  xlabel= {$\omega$},
  ylabel= {$S_V(\omega)$}]
  \addplot[thick, color=gray,densely dotted, mark size=1pt] table [x index=0, y index=1]{mud05/probab_psd_delta05_tau1_mud0.5.txt};
\addplot[forget plot, each nth point=5, only marks,mark=*,mark size=1pt] table [x index=0, y index=1]{mud05/determ_psd_delta05_tau1_mud0.5.txt};
\addlegendentry{$\tau=1$};
\addplot[thick, color=gray,dashdotted,mark=none,mark size=1pt] table [x index=0, y index=1]{mud05/probab_psd_delta05_tau05_mud0.5.txt};
\addplot[forget plot, each nth point=5, only marks,mark=*,mark size=1pt] table [x index=0, y index=1]{mud05/determ_psd_delta05_tau05_mud0.5.txt};
\addlegendentry{$\tau=2^{-1}$};
\addplot[thick,color=gray,mark size=1pt] table [x index=0, y index=1]{mud05/probab_psd_delta05_tau025_mud0.5.txt};
\addplot[forget plot, each nth point=5, only marks,mark=*,mark size=1pt] table [x index=0, y index=1]{mud05/determ_psd_delta05_tau025_mud0.5.txt};
\addlegendentry{$\tau=2^{-2}$};
\node at (-1,-9) {\textcolor{black}{$\mud=0.5, \delta = 2^{-1}$}};
\end{loglogaxis}
\end{tikzpicture}
\begin{tikzpicture}[scale=0.49]
\begin{loglogaxis}[legend style={at={(0,0)},anchor=south west}, compat=1.3,
  xmin=0.001, xmax=50,ymin=0.,ymax=1.5,
  xlabel= {$\omega$},
  ylabel= {$S_V(\omega)$}]
  \addplot[thick, color=gray,densely dotted,mark size=1pt] table [x index=0, y index=1]{mud1/probab_psd_delta05_tau1_mud1.txt};
\addplot[forget plot, each nth point=5, only marks,mark=*,mark size=1pt] table [x index=0, y index=1]{mud1/determ_psd_delta05_tau1_mud1.txt};
\addlegendentry{$\tau=1$};
\addplot[thick, color=gray,dashdotted,mark=none,mark size=1pt] table [x index=0, y index=1]{mud1/probab_psd_delta05_tau05_mud1.txt};
\addplot[forget plot, each nth point=5, only marks,mark=*,mark size=1pt] table [x index=0, y index=1]{mud1/determ_psd_delta05_tau05_mud1.txt};
\addlegendentry{$\tau=2^{-1}$};
\addplot[thick,color=gray,mark size=1pt] table [x index=0, y index=1]{mud1/probab_psd_delta05_tau025_mud1.txt};
\addplot[forget plot, each nth point=5, only marks,mark=*,mark size=1pt] table [x index=0, y index=1]{mud1/determ_psd_delta05_tau025_mud1.txt};
\addlegendentry{$\tau=2^{-2}$};
\node at (-1,-9) {\textcolor{black}{$\mud=1, \delta = 2^{-1}$}};
\end{loglogaxis}
\end{tikzpicture}

\begin{tikzpicture}[scale=0.49]
\begin{loglogaxis}[legend style={at={(0,0)},anchor=south west}, compat=1.3,
  xmin=0.001, xmax=50,ymin=0.,ymax=1.5,
  xlabel= {$\omega$},
  ylabel= {$S_V(\omega)$}]
 \addplot[thick, color=gray,densely dotted,mark size=1pt] table [x index=0, y index=1]{mud025/probab_psd_delta1_tau1_mud0.25.txt};
\addplot[forget plot, each nth point=5, only marks,mark=*,mark size=1pt] table [x index=0, y index=1]{mud025/determ_psd_delta1_tau1_mud0.25.txt};
\addlegendentry{$\tau=1$};
\addplot[thick, color=gray,dashdotted,mark=none,mark size=1pt] table [x index=0, y index=1]{mud025/probab_psd_delta1_tau05_mud0.25.txt};
\addplot[forget plot, each nth point=5, only marks,mark=*,mark size=1pt] table [x index=0, y index=1]{mud025/determ_psd_delta1_tau05_mud0.25.txt};
\addlegendentry{$\tau=2^{-1}$};
\addplot[thick,color=gray,mark=none,mark size=1pt] table [x index=0, y index=1]{mud025/probab_psd_delta1_tau025_mud0.25.txt};
\addplot[forget plot, each nth point=5, only marks,mark=*,mark size=1pt] table [x index=0, y index=1]{mud025/determ_psd_delta1_tau025_mud0.25.txt};
\addlegendentry{$\tau=2^{-2}$};
\node at (-1,-9) {\textcolor{black}{$\mud=0.25, \delta = 2^{-2}$}};
\end{loglogaxis}
\end{tikzpicture}
\begin{tikzpicture}[scale=0.49]
\begin{loglogaxis}[legend style={at={(0,0)},anchor=south west}, compat=1.3,
  xmin=0.001, xmax=50,ymin=0.,ymax=1.5,
  xlabel= {$\omega$},
  ylabel= {$S_V(\omega)$}]
\addplot[thick, color=gray,densely dotted,mark size=1pt] table [x index=0, y index=1]{mud05/probab_psd_delta1_tau1_mud0.5.txt};
\addplot[forget plot, each nth point=5, only marks,mark=*,mark size=1pt] table [x index=0, y index=1]{mud05/determ_psd_delta1_tau1_mud0.5.txt};
\addlegendentry{$\tau=1$};
\addplot[thick, color=gray,dashdotted,mark=none,mark size=1pt] table [x index=0, y index=1]{mud05/probab_psd_delta1_tau05_mud0.5.txt};
\addplot[forget plot, each nth point=5, only marks,mark=*,mark size=1pt] table [x index=0, y index=1]{mud05/determ_psd_delta1_tau05_mud0.5.txt};
\addlegendentry{$\tau=2^{-1}$};
\addplot[thick,color=gray,mark size=1pt] table [x index=0, y index=1]{mud05/probab_psd_delta1_tau025_mud0.5.txt};
\addplot[forget plot, each nth point=5, only marks,mark=*,mark size=1pt] table [x index=0, y index=1]{mud05/determ_psd_delta1_tau025_mud0.5.txt};
\addlegendentry{$\tau=2^{-2}$};
\node at (-1,-9) {\textcolor{black}{$\mud=0.5, \delta = 2^{-2}$}};
\end{loglogaxis}
\end{tikzpicture}
\begin{tikzpicture}[scale=0.49]
\begin{loglogaxis}[legend style={at={(0,0)},anchor=south west}, compat=1.3,
  xmin=0.001, xmax=50,ymin=0.,ymax=1.5,
  xlabel= {$\omega$},
  ylabel= {$S_V(\omega)$}]
 \addplot[thick, color=gray,densely dotted,mark size=1pt] table [x index=0, y index=1]{mud1/probab_psd_delta1_tau1_mud1.txt};
\addplot[forget plot, each nth point=5, only marks,mark=*,mark size=1pt] table [x index=0, y index=1]{mud1/determ_psd_delta1_tau1_mud1.txt};
\addlegendentry{$\tau=1$};
\addplot[thick, color=gray,dashdotted,mark=none,mark size=1pt] table [x index=0, y index=1]{mud1/probab_psd_delta1_tau05_mud1.txt};
\addplot[forget plot, each nth point=5, only marks,mark=*,mark size=1pt] table [x index=0, y index=1]{mud1/determ_psd_delta1_tau05_mud1.txt};
\addlegendentry{$\tau=2^{-1}$};
\addplot[thick,color=gray,mark size=1pt] table [x index=0, y index=1]{mud1/probab_psd_delta1_tau025_mud1.txt};
\addplot[forget plot, each nth point=5, only marks,mark=*,mark size=1pt] table [x index=0, y index=1]{mud1/determ_psd_delta1_tau025_mud1.txt};
\addlegendentry{$\tau=2^{-2}$};
\node at (-1,-9) {\textcolor{black}{$\mud=1, \delta = 2^{-2}$}};
\end{loglogaxis}
\end{tikzpicture}
\caption{
\color{black}{Double logarithmic plot of the PSD of the process $V^\delta$ for $\tau = 2^{-i}, i=0,1, 2$ on each subfigure. The subfigures appear in the following order from left to right: $\mud  = \frac{1}{4},\frac{1}{2},1$ ($\mus = 1$) and from top to bottom: $\delta = 2^{-j}, j= 0, 1, 2$. Deterministic results (obtained by solving \eqref{eq:expressSvomega} and \eqref{eq:hatphi1}) are in  black dots and Monte Carlo simulation results are in gray. Here {\color{black} $\mathfrak{b}(v) = - v$}.}
\label{figX}}
\end{figure}
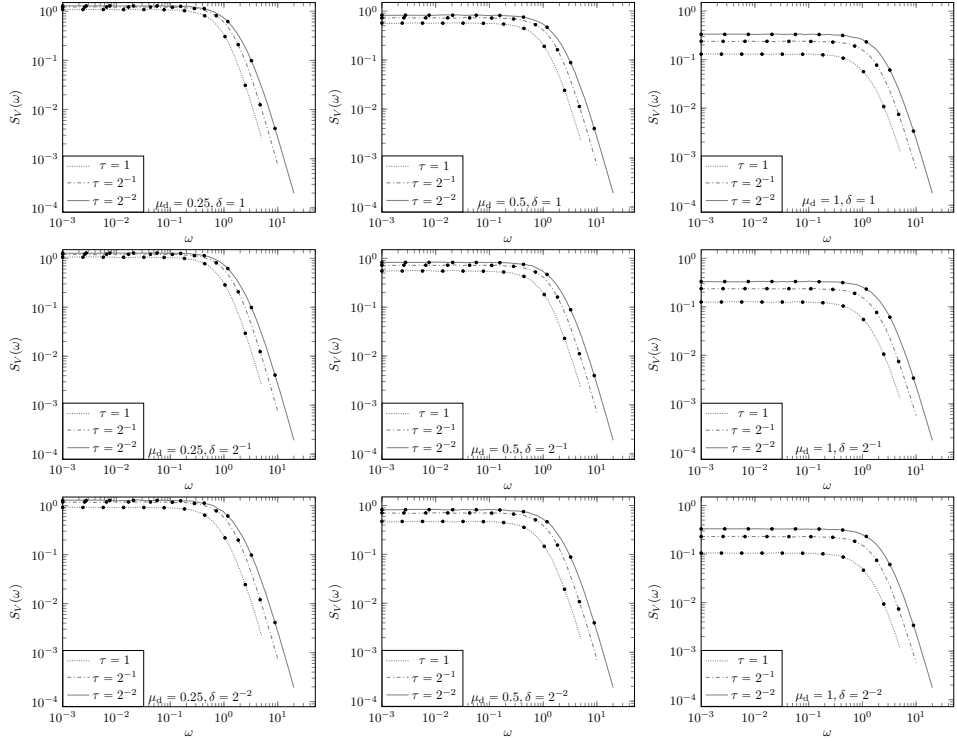
{
\color{black}Since the PSDs have well defined central peaks at $\omega=0$, we consider the full width at half maximum (FWHM) $\Delta \omega$ to define the correlation time of the system denoted by $t_{\rm corr } = 1/ \Delta \omega$. Figure \ref{figY} plots the correlation times of the process $V^\delta$ as functions of $\tau$ in the four cases $\mud \mus^{-1} = \frac{1}{4},\frac{1}{2},\frac{3}{4},1$.   This indicates that for $\tau \leq 0.125$, the correlation time becomes constant (when $\tau$ is small, we may think that the driving process $X$ behaves like a white noise; we then recover the observation that the correlation time essentially coincides with the value of the white-noise limit as long as $ \tau < 0.1$ \cite{GJ2017}).  In the same figure, in the black squares on the left, we observe the correlation time  of the process $V^\delta$  for $\mud=\mus$ and {\color{black} $\mathfrak{b}(v)=0$}. Again, this indicates that for $\tau \leq 0.125$, it becomes constant. }
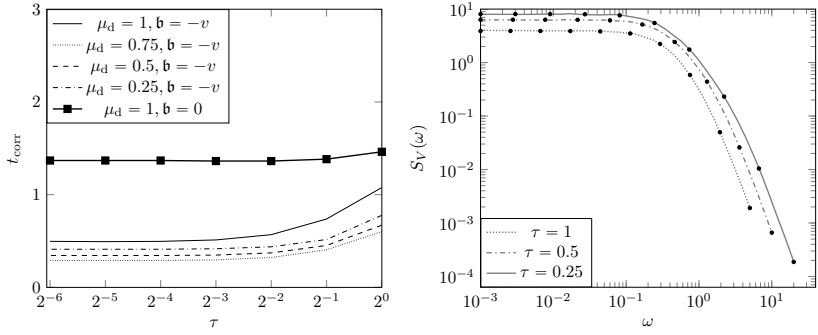
\begin{figure}[h!]
\centering
\begin{tikzpicture}[scale=0.65]
\begin{axis}[legend style={at={(0,1)},anchor=north west}, compat=1.3,
  xmin=0.015, xmax=1,ymin=0,ymax=3,xmode = log, log basis x = {2},
  xlabel= {$\tau$},
  ylabel= {$t_{\textup{corr}}$}]
\addplot[mark=none,mark size=2pt,mark options={color=black}] table [x index=0, y index=1]{tcorr,b=x-v_true.txt};
\addlegendentry{$\mud=1, {\color{black} \mathfrak{b} =-v}$};
\addplot[densely dotted, mark=none,mark size=2pt,mark options={color=black}] table [x index=0, y index=2]{tcorr,b=x-v_true.txt};
\addlegendentry{$\mud=0.75, {\color{black} \mathfrak{b} =-v}$};
\addplot[dashed,mark=none,mark size=2pt,mark options={color=black}] table [x index=0, y index=3]{tcorr,b=x-v_true.txt};
\addlegendentry{$\mud=0.5, {\color{black} \mathfrak{b} =-v}$};
\addplot[dashdotted, mark=none,mark size=2pt,mark options={color=black}] table [x index=0, y index=4]{tcorr,b=x-v_true.txt};
\addlegendentry{$\mud=0.25, {\color{black} \mathfrak{b} =-v}$};
\addplot[thick,mark=square*,mark size=2pt,mark options={color=black}] table [x index=0, y index=1]{tcorr,b=x_true.txt};
\addlegendentry{$\mud=1, {\color{black} \mathfrak{b} = 0}$};
\end{axis}
\end{tikzpicture}
\begin{tikzpicture}[scale=0.65]
\begin{loglogaxis}[legend style={at={(0,0)},anchor=south west}, compat=1.3,
  xmin=0.001, xmax=40,ymin=0.,ymax=10,
  xlabel= {$\omega$},
  ylabel= {$S_V(\omega)$}]
\addplot[thick, color=gray,densely dotted, mark size=1pt] table [x index=0, y index=1]{probab_psd_delta025_tau1_bx.txt};
\addplot[forget plot, only marks,mark=*, mark size=1pt] table [x index=0, y index=1]{determ_psd_delta025_tau1_bx.txt};
\addlegendentry{$\tau=1$};
\addplot[thick,color=gray,dashdotted, mark size=1pt] table [x index=0, y index=1]{probab_psd_delta025_tau05_bx.txt};
\addplot[forget plot, only marks,mark=*, mark size=1pt] table [x index=0, y index=1]{determ_psd_delta025_tau05_bx.txt};
\addlegendentry{$\tau=0.5$};
\addplot[thick, color=gray,mark size=1pt] table [x index=0, y index=1]{probab_psd_delta025_tau025_bx.txt};
\addplot[forget plot, only marks,mark=*, mark size=1pt] table [x index=0, y index=1]{determ_psd_delta025_tau025_bx.txt};
\addlegendentry{$\tau=0.25$};
\end{loglogaxis}
%
\end{tikzpicture}
\caption{
{\color{black}Left: Correlation time of the process $V^\delta$ w.r.t. noise correlation time $\tau$ for $\delta = 2^{-2}$, $\mus=1$, and {\color{black} $\mathfrak{b}(v)=-v$}. The results have been computed with the deterministic method. Right: Double logarithmic plot of the PSD of the process $V^\delta$ for $\tau = 2^{-i}, i=0,1, 2$, $\mud=\mus=1$, and $\delta = 2^{-2}$ with {\color{black} $\mathfrak{b}(v)=0$}. The black dots, resp. the gray lines, come from the deterministic method, resp.  the Monte Carlo method.}
\label{figY}}
\end{figure}

\paragraph{Durations of excursions, static and dynamique phases}
As shown in Figure \ref{fig:histo}, we observe two different behaviors for the pdf $f_{\textup{slide}}$ of the dynamic phase duration.
When $\mud < \mus$, $f_{\textup{slide}}$ vanishes at 0 and is very close to $0$ in its neighborhood.
This indicates the absence of short dynamic phases. 
When $\mud = \mus$, $f_{\textup{slide}}$ vanishes at $0$ but increases very fast. This indicates the presence of dramatically short dynamic phases. 
In all cases, the pdf $f_{\textup{stick}}$ of the static phase duration  is positive around 0 and  $f_{\textup{stick}}(0^+)$ is a finite positive number for fixed $\delta$. This indicates the presence of short static phases.
Finally, the behavior of the excursion is essentially inherited from the behavior of the dynamic phase.
{\color{black}
In Figure \ref{figZ}, we plot the Laplace transform of the duration of the dynamic phase obtained by the probabilistic (Monte Carlo) method and by the deterministic (Kolmogorov) method resulting from Proposition \ref{prop:laplacesliding}. This shows again that both methods give the same results.
}
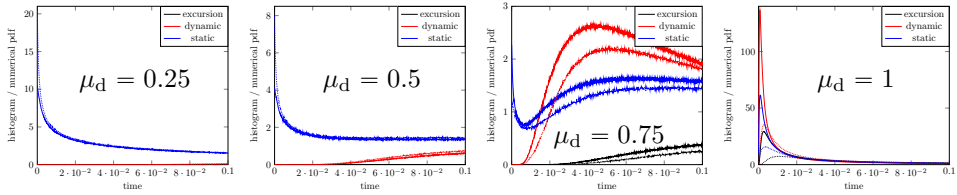
\begin{figure}[h!]
\centering
\begin{tikzpicture}[scale=0.37]
\begin{axis}[legend style={at={(1,1)},anchor=north east}, compat=1.3,
  xmin=0, xmax=0.1,ymin=0,ymax=21,
  xlabel= {time},
  ylabel= {histogram / numerical pdf}] 
\addplot[black,thick,mark size=2pt,mark options={color=red}] table [x index=0, y index=3]{casea_2m4_pdmp_friction_data8.txt}; 
\addlegendentry{excursion};
\addplot[red,thick,mark size=2pt,mark options={color=red}] table [x index=0, y index=2]{casea_2m4_pdmp_friction_data8.txt}; 
\addlegendentry{dynamic};
\addplot[blue,thick,mark size=2pt,mark options={color=red}] table [x index=0, y index=1]{casea_2m4_pdmp_friction_data8.txt}; 
\addlegendentry{static};
\addplot[black,densely dotted,thick,mark size=2pt,mark options={color=red}] table [x index=0, y index=3]{casea_2m5_pdmp_friction_data8.txt}; 
\addplot[red,densely dotted,thick,mark size=2pt,mark options={color=red}] table [x index=0, y index=2]{casea_2m5_pdmp_friction_data8.txt}; 
\addplot[blue,densely dotted,thick,mark size=2pt,mark options={color=red}] table [x index=0, y index=1]{casea_2m5_pdmp_friction_data8.txt}; 
\end{axis}
\node at (3.5cm,3cm) {\textcolor{black}{$\mud=0.25$}};
\end{tikzpicture}
\begin{tikzpicture}[scale=0.37]
\begin{axis}[legend style={at={(1,1)},anchor=north east}, compat=1.3,
  xmin=0, xmax=0.1,ymin=0,ymax=8.5,
  xlabel= {time},
  ylabel= {histogram / numerical pdf}] 
\addplot[black,thick,mark size=2pt,mark options={color=red}] table [x index=0, y index=3]{caseb_2m4_pdmp_friction_data8.txt}; 
\addlegendentry{excursion};
\addplot[red,thick,mark size=2pt,mark options={color=red}] table [x index=0, y index=2]{caseb_2m4_pdmp_friction_data8.txt}; 
\addlegendentry{dynamic};
\addplot[blue,thick,mark size=2pt,mark options={color=red}] table [x index=0, y index=1]{caseb_2m4_pdmp_friction_data8.txt}; 
\addlegendentry{static};
\addplot[black,densely dotted, mark size=2pt,mark options={color=red}] table [x index=0, y index=3]{caseb_2m5_pdmp_friction_data8.txt}; 
\addplot[red,densely dotted,thick,mark size=2pt,mark options={color=red}] table [x index=0, y index=2]{caseb_2m5_pdmp_friction_data8.txt}; 
\addplot[blue,densely dotted,thick,mark size=2pt,mark options={color=red}] table [x index=0, y index=1]{caseb_2m5_pdmp_friction_data8.txt}; 
\end{axis}
\node at (3.5cm,3cm) {\textcolor{black}{$\mud=0.5$}};
\end{tikzpicture}
\begin{tikzpicture}[scale=0.37]
\begin{axis}[legend style={at={(1,1)},anchor=north east}, compat=1.3,
  xmin=0, xmax=0.1,ymin=0,ymax=3,
  xlabel= {time},
  ylabel= {histogram / numerical pdf}] 
  \addplot[black,thick,mark size=2pt,mark options={color=red}] table [x index=0, y index=3]{casec_2m5_pdmp_friction_data8.txt}; 
\addlegendentry{excursion};
\addplot[red,thick,mark size=2pt,mark options={color=red}] table [x index=0, y index=2]{casec_2m5_pdmp_friction_data8.txt}; 
\addlegendentry{dynamic};
\addplot[blue,thick,mark size=2pt,mark options={color=red}] table [x index=0, y index=1]{casec_2m5_pdmp_friction_data8.txt}; 
\addlegendentry{static};
\addplot[black,densely dotted,thick,mark size=2pt,mark options={color=red}] table [x index=0, y index=3]{casec_2m4_pdmp_friction_data8.txt}; 
\addplot[red,densely dotted,thick,mark size=2pt,mark options={color=red}] table [x index=0, y index=2]{casec_2m4_pdmp_friction_data8.txt}; 
\addplot[blue,densely dotted,thick,mark size=2pt,mark options={color=red}] table [x index=0, y index=1]{casec_2m4_pdmp_friction_data8.txt}; 
\end{axis}
\node at (3.5cm,1cm) {\textcolor{black}{$\mud=0.75$}};
\end{tikzpicture}
\begin{tikzpicture}[scale=0.37]
\begin{axis}[legend style={at={(1,1)},anchor=north east}, compat=1.3,
  xmin=0, xmax=0.1,ymin=0,ymax=140,
  xlabel= {time},
  ylabel= {histogram / numerical pdf}] 
  \addplot[black,thick,mark size=2pt,mark options={color=red}] table [x index=0, y index=3]{cased_2m5_pdmp_friction_data8.txt}; 
\addlegendentry{excursion};
\addplot[red,thick,mark size=2pt,mark options={color=red}] table [x index=0, y index=2]{cased_2m5_pdmp_friction_data8.txt}; 
\addlegendentry{dynamic};
\addplot[blue,thick,mark size=2pt,mark options={color=red}] table [x index=0, y index=1]{cased_2m5_pdmp_friction_data8.txt}; 
\addlegendentry{static};
\addplot[black,densely dotted,thick,mark size=2pt,mark options={color=red}] table [x index=0, y index=3]{cased_2m4_pdmp_friction_data8.txt}; 
\addplot[red,densely dotted,thick,mark size=2pt,mark options={color=red}] table [x index=0, y index=2]{cased_2m4_pdmp_friction_data8.txt}; 
\addplot[blue,densely dotted,thick,mark size=2pt,mark options={color=red}] table [x index=0, y index=1]{cased_2m4_pdmp_friction_data8.txt}; 
\end{axis}
\node at (3.5cm,3cm) {\textcolor{black}{$\mud=1$}};
\end{tikzpicture}
\caption{\color{black}{Numerical pdf of the duration of the static phase $f_{\textup{stick}}$ (red), duration of the dynamic phase $f_{\textup{slide}}$ (blue) and duration of an excursion (black). Each curve corresponds to a $\delta$ which takes the values $2^{-4}$ (dotted line) and $2^{-5}$ (solid line). The pdf are numerically determined by Monte Carlo simulations involving $10^7$ excursions. {\color{black} Here $\mathfrak{b}(v)=-v$}.}
\label{fig:histo}}
\end{figure}
\begin{figure}[h!]
\centering
\begin{tikzpicture}[scale=0.6]
\begin{loglogaxis}[legend style={at={(0,0)},anchor=south west}, compat=1.3,
  xmin=0.1, xmax=10,ymin=0.,ymax=1.,
  xlabel= {
  $\lambda$},
  ylabel= {$F_{\rm slide}(\lambda)$
  }]
\addplot[forget plot,only marks, mark=*, mark size=1pt] table [x index=0, y index=2]{resD_delta_1.0000_10.txt};
\addplot[color=gray, thick, mark size=1pt] table [x index=0, y index=2]{probab_laplace_delta1.000_tau1.000_mud1.txt};
\addlegendentry{$\mud=1$};
\addplot[forget plot,only marks, mark=*, mark size=1pt] table [x index=0, y index=2]{resC_delta_1.0000_10.txt};
\addplot[dashed,color=gray, thick, mark size=1pt] table [x index=0, y index=2]{probab_laplace_delta1.000_tau1.000_mud075.txt};
\addlegendentry{$\mud=0.75$};
\addplot[forget plot,only marks, mark=*,mark size=1pt] table [x index=0, y index=2]{resB_delta_1.0000_10.txt};
\addplot[densely dotted, color=gray, thick, mark size=1pt] table [x index=0, y index=2]{probab_laplace_delta1.000_tau1.000_mud05.txt};
\addlegendentry{$\mud=0.5$};
\addplot[forget plot,only marks, mark=*,mark size=1pt] table [x index=0, y index=2]{resA_delta_1.0000_10.txt};
\addplot[dotted, color=gray, thick, mark size=1pt] table [x index=0, y index=2]{probab_laplace_delta1.000_tau1.000_mud025.txt};
\addlegendentry{$\mud=0.25$};
\end{loglogaxis}
\end{tikzpicture}
\caption{
\color{black}{Laplace transform $F_{\rm slide}(\lambda)$ of the dynamic duration. The black dots result from the deterministic method whereas the gray lines from the Monte Carlo method. Here $\delta=1, \tau=1,\mus=1$. {\color{black} Here $\mathfrak{b}(v)=-v$}.}
\label{figZ}}
\end{figure}
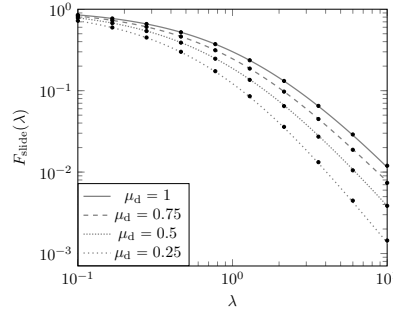

\begin{table}[h!]
\centering
{\scriptsize
\begin{tabular}{||c| r r r||} 
 \hline
 \diagbox{$\mud \mus^{-1}$}{$\delta$} & $2^{-3}$ &  $2^{-4}$ & $2^{-5}$ \\ [0.5ex] 
 \hline\hline
0.25 & 5.32/6.73 & 10.83/10.51 & 21.81/20.94\\ 
0.50 & 2.20/2.78 & 4.26/4.10 & 8.36/7.98\\
0.75 & 0.68/0.89 & 1.24/1.17 & 2.36/2.27\\
1.00 & {0.17/0.18} & 0.41/0.76 & 1.12/9.20\\
\hline
\end{tabular}
}
\caption{Estimation of $f_{\textup{stick}} (0^+)$ using Kolmogorov (left numbers) and Monte Carlo (right numbers) methods. {\color{black} Here $\mathfrak{b}(v)=-v$}.}
\label{table:1}
\end{table}

\paragraph{Comments on the cases $\mud < \mus$}
\noindent Using Monte Carlo and Kolmogorov methods, we can estimate $f_{\textup{stick}} (0^+)$ (shown in Table \ref{table:1}) and $f_{\textup{slide}} (0^+)$. 
For $\mud = 0.25,0.5$ and $0.75$ and for $\delta = 2^{-k}, k = 3 ,\dots ,6$, the computed numbers are positive and finite. Both methods agree qualitatively. In addition, we observe that the calculated values are multiplied by two when the parameter $\delta$ is divided by two. It seems to indicate that $f_{\textup{stick}}(0^+) \uparrow \infty$ as $\delta \downarrow 0$. Besides, both methods indicate that $f_{\textup{slide}} (0^+)=0$. Furthermore the empirical histograms from the MC method reveal that $f_{\textup{slide}} $ vanishes in the neighborhood of $0$. 

\paragraph{Comments on the case $\mud = \mus$}
Using the Kolmogorov method, we can estimate $f_{\textup{stick}} (0^+)$ (shown in the last row of Table \ref{table:1}) for $\delta = 2^{-k},k = 3 ,\dots ,6$, the computed numbers are positive and finite but they do not agree very well with the Monte Carlo method. This may be due to the fact that the slope of $f_{\textup{stick}}$ is significantly steep in the neighborhood of 0. 
We can notice that the MC value overestimates the  value given by the discretized Kolmogorov equations. With the Kolmogorov method, we also show that $f_{\textup{slide}}(0^+)=0$ but the convergence is rather slow in finding the limit of $\lambda F_{\textup{slide}}(\lambda)$ as $\lambda \to +\infty$. With the Monte Carlo method, it is difficult to capture this value directly. For each $\delta = 2^{-k}, k = 3 ,\dots , 6$, the empirical histograms from the MC method indicate that $f_{\textup{slide}}(t)>0$ in the neighborhood of $0$ and $f_{\textup{slide}} (t) \downarrow 0$ as $t \downarrow 0^+$.    

\paragraph{Comment on our memory limit}
The numerical results reported on this work have been performed on a MacBook Air (13-inch, Mid 2013) with the following characterics: Processor 1,3 GHz Intel Core i5, Memory 8 GB 1600 MHz DDR3, Graphics Intel HD Graphics 5000 1536 MB. With such a memory limit, the size of the matrix ${\bf M}$ must remain below $10^7 \times 10^7$. 

\paragraph{A data-learning heuristic to go beyond our memory limit: extrapolation of our results}
Below, we use the notation $S^{[i]}(k,l)$ for the estimation  of $S^i$ defined by (\ref{eq:defSi}) using the deterministic method where $p=2^k$ and $\delta = 2^{-l}$. 
As shown in Figure \ref{fig4}, due to memory limit, we cannot evaluate $S^{[i]}(k,l)$ when $(k,l) \in \bar{\mathfrak{U}} = \{ 7 \leq k \leq 9 \} \times \{ 5 \} \cup  \{ 5 \leq k \leq 9 \} \times \{ 6 \}.  
$
We can only evaluate them when $(k,l) \in \mathfrak{U} = \{ 1 \leq k \leq 9\} \times \{ 1 \leq l \leq 4\} \cup \{ 1 \leq k \leq 6 \} \times \{ 5 \} \cup  \{ 1 \leq k \leq 4 \} \times \{ 6 \}.  
$
Nonetheless, we can cook up an extrapolation approach to estimate the missing data on $\bar{\mathfrak{U}}$ where we keep the notation $S^{[i]}(k,l)$. To compute $S^{[i]}(k,l)$ on $\bar{\mathfrak{U}}$, we assume that the error trends observed when $\delta$ is large or $p$ small remain the same as when $\delta$ is small and $p$ large. Our \textit{heuristics} starts from this observation
\begin{align}
& S^{[i]}(k,l) = S^{[i]}(k,l-1) + R_V(k,l) \left ( S^{[i]}(k,l-1) - S^{[i]}(k,l-2) \right ) , \\
& S^{[i]}(k,l) = S^{[i]}(k-1,l) + R_H(k,l) \left ( S^{[i]}(k-1,l) - S^{[i]}(k-2,l) \right )  ,
\end{align}
with 
\begin{equation}
R_V(k,l) = \frac{S^{[i]}(k,l)-S^{[i]}(k,l-1)}{S^{[i]}(k,l-1) - S^{[i]}(k,l-2)} \: \mbox{ and } \: R_H(k,l) = \frac{S^{[i]}(k,l)-S^{[i]}(k-1,l)}{S^{[i]}(k-1,l) - S^{[i]}(k-2,l)}.
\end{equation}
Clearly, $R_V(k,l)$ and $R_H(k,l)$ are unknown since they depend on $S^{[i]}(k,l)$, the targeted unknown quantity. However, we have an idea of the error trend and then our heuristics consists in the following natural approximation $R_V(k,l) \approx R_V(k,l-1)$ and $R_V(k,l) \approx R_V(k-1,l)$.
Then, we define
 \begin{align}
& S_V^{[i]}(k,l) = S^{[i]}(k,l-1) + R_V(k,l-1) \left ( S^{[i]}(k,l-1) - S^{[i]}(k,l-2) \right ) , \\
& S_H^{[i]}(k,l) = S^{[i]}(k-1,l) + R_H(k-1,l) \left ( S^{[i]}(k-1,l) - S^{[i]}(k-2,l) \right ) ,
\end{align}
and finally
\begin{equation}
S^{[i]}(k,l)  = \frac{1}{2} \left( S_V^{[i]}(k,l) + S_H^{[i]}(k,l) \right ).
\end{equation}
In this way, this heuristical definition of $S^{[i]}(k,l)$ requires six values: 
$S^{[i]}(k-j,l)$ and $S^{[i]}(k,l-j)$, $j=1,2,3$.
Starting from the data on $\mathfrak{U}$, we can propagate this data-learning scheme on $\bar{\mathfrak{U}}$.
In Table \ref{tab:heuristics}, we present our results which show the stability of the extrapolation procedure.
\begin{table}[h!]
\centering
{\scriptsize
\begin{tabular}{||c c c c c c c c c||} 
 \hline
 & (7,5) & (8,5) & (9,5) & (5,6) & (6,6) & (7,6) & (8,6) & (9,6) \\ 
 \hline\hline
 $S^{[1]}$ & 0.1291 & 0.1294 & 0.1295 (0.1302) & 0.1239 & 0.1244 & 0.1249 & 0.1254 & 0.1258 (0.1257)\\
 $S^{[2]}$ & 0.6886 & 0.6893 & 0.6896 (0.6980) & 0.6803 & 0.6835 & 0.6850 & 0.6855 & 0.6857 (0.6850) \\
 \hline
\end{tabular}
}
\caption{Computation of $S^{[i]}(k,l)$ with our heuristics on $\bar{\mathfrak{U}}$. The first row lists the elements of $\bar{\mathfrak{U}}$. The values in parenthesis are shown for comparison and have been obtained with different methods, Monte Carlo in the $S^{[1]}$ row and (stationary) Kolmogorov equation for ${\color{black}X}$ in the $S^{[2]}$ row. {\color{black} Here $\mathfrak{b}(v)=-v$}.
}
\label{tab:heuristics}
\end{table}

\section{Conclusions and perspectives}
In this work, we tackle the problem of modeling stochastic dry friction including different coefficients for the static and dynamic forces by proposing a PDMP approach. Here the external forcing takes discrete values and it is assumed to be a Markov jump process depending on a small parameter.
We show ergodicity and provide a representation formula of the stationary measure.
We also obtain a characterization of the Laplace transforms of the probability density functions of the durations of the static and dynamic phases. Moreover, when the aforementioned parameter vanishes and when the two coefficients of static and dynamic forces are identical, we show that the PDMP converges in distribution to the solution of a well-known dry friction model. This model is subjected
to a colored noise (an Ornstein-Uhlenbeck process) and its definition involves a differential inclusion formalism. This bridges the
gap between our approach and existing well-posed continuous models when the coefficients for the static and dynamic forces are identical.
As a future work, it should be possible to consider the extension of the PDMP approach to higher dimensions and more realistic systems such as randomly driven moveable rigid bodies in frictional contact with rigid obstacles or mechanical systems of rigid bodies as inspired by \cite{siam1Lotstedt,Erdmann94}. 
It would also be of interest to develop numerical simulation and stochastic control methods for these systems by combining existing techniques such as Lagrange multipliers in the same spirit as \cite{siam2Lotstedt} and \cite{siamRolandGlowinski}.

\section{Acknowledgements}
The authors are grateful  to  anonymous referees  for  their remarks and suggestions which were very helpful in improving the manuscript.

 \appendix
 
\section{Proof of Proposition \ref{prop:1}}
\label{app:prop1}
This is a diffusion approximation result.

First, the process ${\color{black}X}^\delta_t$ is Markov with the generator $Q^\delta {\color{black}f}({\color{black}x}) = 2\delta^{-2}\tau^{-1} \big( \alpha({\color{black}x}){\color{black}f}({\color{black}x}+\delta)+\alpha^\star({\color{black}x}) {\color{black}f}({\color{black}x}-\delta) - {\color{black}f}({\color{black}x})\big)$ and it converges in distribution in the space of the c\`adl\`ag functions to the diffusion process with generator $Q$. Indeed $Q^\delta {\color{black}f}({\color{black}x})=Q{\color{black}f}({\color{black}x})+o(1)$ for any smooth test function ${\color{black}f}$ and $Q=\tau^{-2}  \partial_{\color{black}x}^2 - \tau^{-1} {\color{black}x}\partial_{\color{black}x} $ is the infinitesimal
generator of ${\color{black}X}_t$ \cite[Chapter 12]{MR0838085}.

Second the map ${\color{black}X}^\delta \mapsto {\color{black}V}^\delta$ from the space of the c\`adl\`ag functions to the space of the continuous functions, with $v$ solution of  $\dot {\color{black}V}^\delta  + \partial \varphi ({\color{black}V}^\delta) \ni { \color{black} \mathfrak{b}(V^\delta) + X^\delta }$, is continuous. We now present the proof of this statement (formulated in Proposition \ref{prop:B2}).

\subsection*{Notation and assumption} 
The set of (real valued) right continous left limit (c\`adl\`ag) functions on $[0,T]$ is denoted by $D[0,T]$.
The set of continuous functions on $[0,T]$ is denoted by $C[0,T]$. Clearly $C[0,T] \subset D[0,T]$.
We consider $\mathfrak{b}: \mathbb{R} \to \mathbb{R}$ a Lipschitz function of Lipschitz constant $L_b>0$, $\xi \in \mathbb{R}$ and $\varphi(v)  = \mu |v|, \mu >0$.
\paragraph{Converging sequence in the $J_1$ topology \cite{bill}}
We say that a sequence of functions $\{ w_n \} \in D[0,T]$ converges towards a function $w \in D[0,T]$ in the sense of $J_1$ topology if 
there exists a sequence of increasing homeomorphims $\{ \lambda_n \}$ on $[0,T]$ such that $\lambda_n(0) = 0, \: \lambda_n(T) = T$ and
\begin{equation}
\textbf{(a)} \lim \limits_{n \to \infty} \sup \limits_{0 \leq t \leq T} |\lambda_n(t) - t| = 0
\quad
\mbox{and}
\quad
\textbf{(b)} \lim \limits_{n \to \infty} \sup \limits_{0 \leq t \leq T} |w_n(\lambda_n(t)) - w(t)| = 0.
\label{eq:defconvJ1}
\end{equation}

\subsection*{Preliminary : case of a differential equation with a c\`adl\`ag function at the rhs}
Let $w \in D[0,T]$ and $\xi \in \RR$.
Consider the following problem:
\begin{equation}
\label{Px}
\begin{cases}
&
\textup{find a function \,} \: v(w) \in C[0,T] \:  \textup{ satisfying } \:   \\
& \forall t \geq 0, \: v_t(w) = \xi + \int_0^t \mathfrak{b}(v_s(w)) \textup{d} s + \int_0^t w(s) \textup{d} s.
\end{cases}
\end{equation}
\begin{prop}
\label{prop:B1}
There exists a unique solution to the problem \eqref{Px}.
As a consequence, the mapping $v$ which associates $w$ to $v(w)$ from $D[0,T]$ to $C[0,T]$ is well defined. 
Moreover, $v$ is continous with respect to the $J_1$ topology on $D[0,T]$ in the sense that if a sequence of functions $w_n \in D[0,T]$ converges to a function $w \in D[0,T]$ as $n \to \infty$ then $v(w_n)$ converges to $v(w)$ as $n \to \infty$ in $C[0,T]$.
\end{prop}
\begin{proof} \underline{Part 1}.
The existence of a solution can be obtained by Picard's iteration.  
First define $\forall t \geq 0, \: v_t^0(w) \equiv \xi$ and then
$$
\forall n \geq 0, \: \forall t \geq 0, \: v_t^{n+1}(w) = \xi + \int_0^t \mathfrak{b}(v_s^n(w)) \textup{d} s + \int_0^t w(s) \textup{d} s.
$$
The sequence $\{ v^n(w)\}$ is composed of continuous functions. 
Since $\mathfrak{b}$ is Lipschitz it converges uniformly on $[0,T]$.
The limit is denoted by $v(w)$ and it satisfies \eqref{Px}.\\
\underline{Part 2}.
 If $v(w)$ and $\tilde v(w)$ are two solutions of \eqref{Px} then we must have 
$$
\sup \limits_{0 \leq r \leq t} |v_r(w)-\tilde v_r (w)| \leq L_b \int_0^t \sup \limits_{0 \leq r \leq s} |v_r(w)-\tilde v_r (w)| \textup{d}s  ,
$$
which implies, by Gronwall's lemma, that $v(w)=\tilde v(w)$ in $C[0,T]$.\\
\underline{Part 3.}
Let  $\{ w_n \}$ be a sequence of functions in $D[0,T]$ converging towards a function $w \in D[0,T]$ in the $J_1$ topology.
Since $\forall t \geq 0$, 
$$ 
v_t(w_n) = \xi + \int_0^t \mathfrak{b}(v_s(w_n)) \textup{d} s + \int_0^t w_n(s) \textup{d} s
\: \mbox{ and } \:  
v_t(w) = \xi + \int_0^t \mathfrak{b}(v_s(w)) \textup{d} s + \int_0^t w(s) \textup{d} s, 
$$
we have
 $$
\sup \limits_{0 \leq r \leq t} |v_r(w)- v_r (w_n )| 
\leq L_b \int_0^t \sup \limits_{0 \leq r \leq s} |v_r(w)- v_r (w_n)| \textup{d} s + \int_0^t |w(s)-w_n(s)| \textup{d} s.
$$
The latter implies using Gronwall's lemma that 
 $$
\sup \limits_{0 \leq r \leq T} |v_r(w)- v_r (w_n )| 
\leq  \exp(L_b T) \int_0^T |w(s)-w_n(s)| \textup{d} s.
$$
Finally, we verity that 
$$
\int_0^T |w(s)-w_n(s)| \textup{d} s \to 0 \: \mbox{ as } \: n \to \infty.
$$
Indeed 
$
\int_0^T |w(s)-w_n(s)| \textup{d} s \leq A_n+B_n$, where $A_n= \int_0^T |w(s)-w(\lambda_n^{-1}(s))| \textup{d} s$, $B_n= \int_0^T |w(\lambda_n^{-1}(s))-w_n(s)| \textup{d} s$, and
$\lambda_n$ is a sequence  of increasing homeomorphims associated to the convergence of $w_n$ in the $J_1$ sense. 
We have $B_n \leq \| w-w_n\circ \lambda_n\|_\infty T$ which shows that $\lim \limits_{n \to \infty }B_n = 0$ by (\ref{eq:defconvJ1}b).
We also have
$w(\lambda_n^{-1}(s)) \to w(s)$ at any point of continuity of $w$, that is to say, almost surely (with respect to the Lebesgue measure over $[0,T]$), and $w(s) -w(\lambda_n^{-1}(s))$ is bounded by $2\|w\|_\infty$, so that $A_n \to 0$
by dominated convergence.
\end{proof}


\subsection*{Case of a differential inclusion with a c\`adl\`ag function at the rhs}
Let $w \in D[0,T]$ and $\xi \in \RR$.
Consider the following problem 
\begin{equation}
\label{Pv}
\begin{cases}
& \textup{find a function} \: v(w) \in C[0,T] \: \textup{ satisfying } \: \\
& \forall t \geq 0, \: v_t(w) + \Delta_t(w) = \xi + \int_0^t \mathfrak{b}(v_s(w)) \textup{d} s + \int_0^t w(s) \textup{d} s , 
\end{cases}
\end{equation}
where $\Delta(w) \in H^1(0,T)$ and with the notation $\delta(w) = \dot \Delta(w)$ 
\begin{align*}
& \forall \zeta \in C[0,T], \forall 0 \leq t < t+h \leq T,\\
& \int_t^{t+h} \left ( \delta_s(w)(\zeta(s) - v_s(w)) + \varphi(v_s(w)) \right )\textup{d} s  \leq \int_t^{t+h} \varphi(\zeta(s)) \textup{d} s.  
\end{align*}
The conditions in \eqref{Pv} are encoded in the differential inclusion notation
$$
\dot v + \partial \varphi(v) \ni \mathfrak{b}(v) + w.
$$
{
\color{black}
\begin{remark}
The mathematical problem which consists in finding a continuous and a.e. differentiable function $v(.)$ satisfying (A.3) is well posed. Its solution describes the velocity of an object subject to Coulomb friction when $\mu = \mud=\mus$. Indeed, when $v=0$ on a non empty time interval then $\dot v = 0$ and necessarily $|\mathfrak{b}(v) + w| \leq \mu$. It is a static phase for $v$.
If $\pm v >0$ (which occurs on non empty time interval) then 
$\dot v \pm \mu = \mathfrak{b}(v)$. It is a dynamic phase for $v$. The multivalued operator $\partial \varphi$ governs the phase transitions at which loss of differentiability of $v$ may occur. These two phases correspond to those mentioned for describing dry friction in the introduction above Equation (1.1). For any $t\geq 0$, the multivalued operator $\partial \varphi$ applied to $v(t)$ is the set of sub-slopes of $\varphi$ in $v(t)$ 
$$
\partial \varphi(v(t)) = \{ \mathfrak{q} \in \mathbb{R}, \: \forall \zeta \in \mathbb{R}, \: \mathfrak{q}(\zeta-v(t)) + \varphi(v(t)) \leq \varphi(\zeta) \}.
$$ 
Therefore, it is possible to formulate (A.3) under the form of a variational inequality
$$
\forall \: \mbox{a.e.} \: t \geq 0, 
\: \forall \zeta \in \mathbb{R}, \: (\dot v(t) - \mathfrak{b}(v(t)) - w(t))(\zeta - v(t)) + \varphi(v(t)) \leq \varphi(\zeta),  
$$  
here the role of $\zeta$ is to act as a real valued test parameter. Furthermore, as $\dot v(.) \in L_{\textup{loc}}^2$ is only defined a.e., it is convenient to work with an integrated version on arbitrary small intervals. In this case, the test parameter $\zeta$ becomes a real valued continuous test function. 
\end{remark}}
\begin{prop}
\label{prop:B2}
There exists a unique solution to the problem \eqref{Pv}.
As a consequence, the mapping $v$ which associates $w$ to $v(w)$ from $D[0,T]$ to $C[0,T]$ is well defined. 
Moreover, $v$ is continous with respect to the $J_1$ topology on $D[0,T]$.
\end{prop}
\begin{proof}
\underline{Part 1} The proof follows the steps of the one of  \cite[proposition C.1]{cpam_long} which addresses the same problem when $w \in C[0,T]$. We recall the essential steps.
For any $p$ we denote by $\varphi_p$ the Moreau-Yosida regularization of $\varphi$,
$$
\varphi_p(v) = 
\begin{cases}
|v| - \frac{1}{2p}& \: |v| > \frac{1}{p}\\
p \frac{v^2}{2} & \: |v| \leq \frac{1}{p} 
\end{cases}
$$
and we consider the penalized problem 
$$
\forall t \geq 0, \: v_t^p(w) + \int_0^t \varphi_p'(v_s^p(w)) \textup{d} s = \xi + \int_0^t \mathfrak{b}(v_s^p(w)) \textup{d} s + \int_0^t w(s) \textup{d} s.
$$
From Proposition \ref{prop:B1}, this ODE has a unique solution $v^p(w) \in C[0,T]$. 
It can be shown that $\{ v^p(w) \}$ is a Cauchy sequence in $C[0,T]$ and  
$$
\sup \limits_{0 \leq t \leq T} |v_t^p(w) - v_t^q(w) |^2 \leq \left ( \frac{1}{p} + \frac{1}{q} \right ) C_T  ,
$$
where the constant $C_T$ depends only on the Lipschitz constant of $\mathfrak{b}$, $T$ and $\mu$.
It is a consequence of the property 
$
\sup \limits_{p \geq 1} \sup \limits_{ v \in \mathbb{R}} |\varphi_p'(v)| = \mu.
$
Thus the limit $v(w) \in C[0,T]$ exists and satisfies 
$$
\sup \limits_{0 \leq t \leq T} |v_t^p(w) - v_t(w) | \leq \sqrt{\frac{C_T}{p}} .
$$ 
It can then be shown that $v(w)$ satisfies the conditions in \eqref{Pv}.\\
\underline{Part 2.}
Assume $w_n \in D[0,T]$ converges to a function $w \in D[0,T]$ as $n \to \infty$. We want to show that then $v(w_n)$ converges to $v(w)$ as $n \to \infty$ in $C[0,T]$.
We can write
\begin{align*}
\sup \limits_{0 \leq t \leq T} | v_t(w_n) - v_t(w) | 
\leq & 
\sup \limits_{0 \leq t \leq T} | v_t(w_n) - v_t^p(w_n) | + 
\sup \limits_{0 \leq t \leq T} | v_t^p(w_n) - v_t^p(w) |\\
& + \sup \limits_{0 \leq t \leq T} | v_t^p(w) - v_t(w) |.  
\end{align*}
Let $\varepsilon >0$. For $p$ large enough we have
$$
\sup \limits_{0 \leq t \leq T} | v_t^p(w) - v_t(w) | \leq \frac{\varepsilon}{3}
\:
\mbox{ and }
\:
\sup \limits_{n} \sup \limits_{0 \leq t \leq T} | v_t(w_n) - v_t^p(w_n) | \leq \frac{\varepsilon}{3}  .
$$
Finally for $n \geq n_p$ large enough 
$$
\sup \limits_{0 \leq t \leq T} | v_t^p(w_n) - v_t^p(w) | \leq \frac{\varepsilon}{3},
$$
which completes the proof of the proposition.
\end{proof}

\section{Proof of Proposition \ref{prop:31}}
\label{app:a}
We want to establish that ${\tau_1}$ is integrable, 
$\EE_{\mathfrak{s}_+}   [ {\tau_1}  ] <+\infty$ (which also proves by symmetry that $\EE_{\mathfrak{s}_-}   [ {\tau_1}  ]  = \EE_{\mathfrak{s}_+}   [ {\tau_1}  ]  <+\infty$).

{\color{black}The process $V_t^\delta$ is bounded by $\max(|V_0^\delta|, v_\textup{max})$. }

{\it Step 1.
Let $\tilde{\tau}_1 = \inf \{ t >0, \,  {\color{black}V}^\delta_t = 0\}$. We have 
$$
C_{\tilde{\tau}} := \sup_{{\color{black}x} \in \{{\color{black}x}_{k_\mus+1},\ldots,{\color{black}x}_N\}} \EE_{({\color{black}x},1,0)} [ \tilde{\tau}_1 ] <+\infty.
$$
}
By symmetry we have $\EE_{({\color{black}x},-1,0)} [ \tilde{\tau}_1 ] =\EE_{(-{\color{black}x},1,0)} [ \tilde{\tau}_1 ]$ for ${\color{black}x} \in \{{\color{black}x}_{-N},\ldots,{\color{black}x}_{-k_\mus-1}\}$, and therefore $ \sup_{{\color{black}x}\in \{{\color{black}x}_{-N},\ldots,{\color{black}x}_{-k_\mus-1}\}} \EE_{({\color{black}x},-1,0)} [ \tilde{\tau}_1 ]=C_{\tilde{\tau}}$.\\
{\it Proof.}
If ${\color{black}X}^\delta_0 \in \{{\color{black}x}_{k_\mus+1},\ldots,{\color{black}x}_N\}$, ${\color{black}V}^\delta_0=0$ and $t< \tilde{\tau}_1 $, then 
{\color{black}$
0\leq {\color{black}V}^\delta_t = \int_0^t [-\mud+b({\color{black}X}^\delta_s ,{\color{black}V}^\delta_s) ] \textup{d} s \leq -\mud t + \int_0^t {\color{black}X}^\delta_s \textup{d} s
$}.
Therefore, for any ${\color{black}x} \in \{{\color{black}x}_{k_\mus+1},\ldots,{\color{black}x}_N\}$ and $t>0$, we have
\begin{align*}
\PP_{({\color{black}x},1,0)} \big( \tilde{\tau}_1 > t \big) 
= \PP_{({\color{black}x},1,0)}\big( \tilde{\tau}_1 > t, \, {\color{black}V}^\delta_t \geq 0 \big)
&\leq 
\PP_{({\color{black}x},1,0)}\big(\int_0^t {\color{black}X}^\delta_s \textup{d} s \geq \mud t 
 \big)
  \\
 &
 \leq \mud^{-4} t^{-4} \EE_{\color{black}x} \big[ \big(\int_0^t {\color{black}X}^\delta_s \textup{d} s\big)^4 \big].
\end{align*}
By the ergodic properties of $({\color{black}X}_t^\delta)$, we have $ t^{-2} \EE_{\color{black}x} \big[ \big(\int_0^t {\color{black}X}^\delta_s \textup{d} s\big)^4 \big] 
\stackrel{t \to +\infty}{\longrightarrow} 6 \big( \int_0^{+\infty} \EE_{{\rm s}} [{\color{black}X}_0^\delta {\color{black}X}_s^\delta] \textup{d} s\big)^2$ which is finite (where $\EE_{\rm s}$ is the expectation under the stationary distribution of the process ${\color{black}X}_t^\delta$).
This shows that there exists $C_{\mud}>0$ such that, for all $t>0$,
$$
\sup_{ {\color{black}x} \in \{{\color{black}x}_{k_\mus+1},\ldots,{\color{black}x}_N\}}
\PP_{({\color{black}x},1,0)} \big( \tilde{\tau}_1 > t \big) \leq  \frac{C_{\mud}}{1+t^2}  ,
$$
which gives 
the desired result.
\qed

{\it Step 2. We have 
$$
C_{p} := \inf_{ {\color{black}x} \in \{{\color{black}x}_{k_\mus+1},\ldots,{\color{black}x}_N\}} \PP_{({\color{black}x},1,0)}\big( {\color{black}X}^\delta_{\tilde{\tau}_1} \in \{ -{\color{black}x}_{k_\mus},\ldots,{\color{black}x}_{k_\mus}\}\big) > 0 .
$$
}
{\it Proof.} 
Let $k_{\rm d}$ be the largest index such that ${\color{black}x}_{k_{\rm d}}<\mud$.
We here denote by $\theta_j$ the times between two random jumps of the process ${\color{black}X}^\delta$ and by $X_j \in \{-\delta,\delta\}$ the jump amplitudes. For $k \in \{ k_\mus+1 ,\ldots, N\}$, we consider 
$B_k = \{  \theta_1+\cdots+\theta_{k-k_{\rm d}} <1, \, X_{k-k_{\rm d}}=\cdots =X_1=-\delta , \theta_{k-k_{\rm d}+1}> ({\color{black}x}_k-\mud)/(\mud-{\color{black}x}_{k_{\rm d}})\}$. This corresponds to a trajectory that goes northwest from $({\color{black}x}_k,1,0)$ up to the line $({\color{black}x}_{k_{\rm d}},1,*)$ in time less than $1$, and then goes south until reaching $({\color{black}x}_{k_{\rm d}},1,0)$ which triggers a deterministic jump to $({\color{black}x}_{k_{\rm d}},0,0)$.
We have $\PP_{({\color{black}x}_k,1,0)}(B_k) >0$ and 
$\PP_{({\color{black}x}_k,1,0)}\big( {\color{black}X}^\delta_{\tilde{\tau}_1} \in \{ {\color{black}x}_{-k_\mus},\ldots,{\color{black}x}_{k_\mus}\}\big) \geq \PP_{({\color{black}x}_k,1,0)}(B_k) >0$,
which gives the desired result after taking $\inf_{k \in \{ k_\mus+1,\ldots, N\}}$.\qed

{\it Step 3. 
 We have 
$$
C_{\hat{\tau}} := \sup_{{\color{black}x} \in \{{\color{black}x}_{k_\mus+1},\ldots,{\color{black}x}_N\}} \EE_{({\color{black}x},1,0)} [ \hat{\tau}_1 ] <+\infty.
$$
}
By symmetry we have $\EE_{({\color{black}x},-1,0)} [ \hat{\tau}_1 ] =\EE_{(-{\color{black}x},1,0)} [ \hat{\tau}_1 ]$ for ${\color{black}x} \in \{{\color{black}x}_{-N},\ldots,{\color{black}x}_{-k_\mus-1}\}$, and therefore $ \sup_{ {\color{black}x} \in \{{\color{black}x}_{-N},\ldots,{\color{black}x}_{-k_\mus-1}\}} \EE_{({\color{black}x},-1,0)} [ \hat{\tau}_1 ]=C_{\hat{\tau}}$.\\
{\it Proof.}
Using the strong Markov property, we have for ${\color{black}x} \in \{{\color{black}x}_{k_\mus+1},\ldots,{\color{black}x}_N\}$:
\begin{align*}
\EE_{({\color{black}x},1,0)} [\hat{\tau}_1] 
&= \EE_{({\color{black}x},1,0)} \big[ \hat{\tau}_1 {\mathbf 1}_{ {\color{black}X}^\delta_{\tilde{\tau}_1} \in \{ -{\color{black}x}_{k_\mus},\ldots,{\color{black}x}_{k_\mus}\} } \big] 
+ \EE_{({\color{black}x},1,0)}   \big[ \hat{\tau}_1 {\mathbf 1}_{ {\color{black}X}^\delta_{\tilde{\tau}_1} \not\in \{ -{\color{black}x}_{k_\mus},\ldots,{\color{black}x}_{k_\mus} \}} \big] \\
&= \EE_{({\color{black}x},1,0)} \big[ \tilde{\tau}_1 {\mathbf 1}_{ {\color{black}X}^\delta_{\tilde{\tau}_1} \in \{ -{\color{black}x}_{k_\mus},\ldots,{\color{black}x}_{k_\mus}\} } \big] 
+ \EE_{({\color{black}x},1,0)}   \big[ (\hat{\tau}_1 -\tilde{\tau}_1 +\tilde{\tau}_1 ) {\mathbf 1}_{ {\color{black}X}^\delta_{\tilde{\tau}_1} \not\in \{ -{\color{black}x}_{k_\mus},\ldots,{\color{black}x}_{k_\mus} \}} \big] \\
&= \EE_{({\color{black}x},1,0)} \big[ \tilde{\tau}_1 \big] 
+ \EE_{({\color{black}x},1,0)}   \big[ (\hat{\tau}_1 -\tilde{\tau}_1){\mathbf 1}_{ {\color{black}X}^\delta_{\tilde{\tau}_1} \not\in \{ -{\color{black}x}_{k_\mus},\ldots,{\color{black}x}_{k_\mus} \}} \big] \\
&= \EE_{({\color{black}x},1,0)} \big[ \tilde{\tau}_1 \big] +
\sum_{\tilde{{\color{black}x}} \in  \{{\color{black}x}_{k_\mus+1},\ldots,{\color{black}x}_N\} }
 \EE_{(\tilde{{\color{black}x}},1,0)} \big[ \hat{\tau}_1  \big] \PP_{({\color{black}x},1,0)} \big({\color{black}X}^\delta_{\tilde{\tau}_1}=\tilde{{\color{black}x}}\big) \\
 &\qquad +
\sum_{\tilde{{\color{black}x}} \in  \{{\color{black}x}_{-N},\ldots,{\color{black}x}_{-k_\mus-1}\} }
 \EE_{(\tilde{{\color{black}x}},-1,0)} \big[ \hat{\tau}_1  \big] \PP_{({\color{black}x},1,0)} \big({\color{black}X}^\delta_{\tilde{\tau}_1}=\tilde{{\color{black}x}}\big) \\
 &\leq 
 C_{\tilde{\tau}} + \sup_{\tilde{{\color{black}x}} \in \{  {\color{black}x}_{k_\mus+1},\ldots,{\color{black}x}_N \}} \EE_{(\tilde{{\color{black}x}},1,0)} \big[ \hat{\tau}_1  \big] (1-C_p) ,
\end{align*}
hence $\sup_{{\color{black}x} \in \{  {\color{black}x}_{k_\mus+1},\ldots,{\color{black}x}_N \}} \EE_{({\color{black}x},1,0)} \big[ \hat{\tau}_1  \big] \leq C_{\tilde{\tau}}/C_p$.
\qed

{\it Step 4.  We have 
$$
C_{{\tau}} := \sup_{{\color{black}x} \in \{{\color{black}x}_{k_\mus+1},\ldots,{\color{black}x}_N\}} \EE_{({\color{black}x},1,0)} [ \tau_1 ] <+\infty.
$$
}
By symmetry we have $\EE_{({\color{black}x},-1,0)} [{\tau}_1 ] =\EE_{(-{\color{black}x},1,0)} [ {\tau}_1 ]$ for ${\color{black}x} \in \{{\color{black}x}_{-N},\ldots,{\color{black}x}_{-k_\mus-1}\}$, and therefore $ \sup_{ {\color{black}x} \in \{{\color{black}x}_{-N},\ldots,{\color{black}x}_{-k_\mus-1}\}} \EE_{({\color{black}x},-1,0)} [ {\tau}_1 ]=C_{{\tau}}$.\\
{\it Proof.}
Using the strong Markov property, we have for ${\color{black}x} \in \{{\color{black}x}_{k_\mus+1},\ldots,{\color{black}x}_N\}$:
\begin{align*}
\EE_{({\color{black}x},1,0)} [ \tau_1 ] = \EE_{({\color{black}x},1,0)} [ \hat{\tau}_1 ]  +  \EE_{({\color{black}x},1,0)} [ {\tau}_1 - \hat{\tau}_1 ] 
\leq C_{\hat{\tau}} + \sup_{\hat{{\color{black}x}} \in \{-{\color{black}x}_{k_\mus},\ldots,{\color{black}x}_{k_\mus}\} } \EE_{(\hat{{\color{black}x}},0,0)} [ \check{\tau}_1 ]  ,
\end{align*}
where $\check{\tau}_1 = \inf \{ t>0, \, {\color{black}X}_t^\delta \in \{{\color{black}x}_{-k_\mus-1},{\color{black}x}_{k_\mus+1}\}\}$. Since the process ${\color{black}X}^\delta_t$ is ergodic,
we have  $ \EE_{(\hat{{\color{black}x}},0,0)} [ \check{\tau}_1 ]<+\infty$ for all $\hat{{\color{black}x}} \in \{-{\color{black}x}_{k_\mus},\ldots,{\color{black}x}_{k_\mus}\} $.
\qed

\section{Proof of Proposition \ref{prop:32}}
\label{app:b}
We introduce an auxiliary Markov process $({\color{black}X}_t^\epsilon,{\color{black}Y}_t^\epsilon,{\color{black}V}_t^\epsilon)$ 
that depends on an additional time parameter $\epsilon>0$:
\begin{itemize}
\item
From $\mathfrak{s}_\pm'$  the process $({\color{black}X}_t^\epsilon,{\color{black}Y}_t^\epsilon,{\color{black}V}_t^\epsilon)$ moves  to $\mathfrak{s}_\pm=\pm ({\color{black}x}_{k_\mus +1},1,0)$ with probability one.
The exponential time of jump has mean $\epsilon \tau^2 \delta^2$.
\item
From $({\color{black}x}_{k_\mus},0,0)$  the process $({\color{black}X}_t^\epsilon,{\color{black}Y}_t^\epsilon,{\color{black}V}_t^\epsilon)$ moves  to $\mathfrak{s}_+'$
with probability  $\alpha({k_\mus})$
and to $({\color{black}x}_{k_\mus-1},0,0)$ with probability $1-\alpha({k_\mus})$.
The exponential time of jump has mean $\tau^2 \delta^2$.
\item
From $({\color{black}x}_{-k_\mus},0,0)$  the process $({\color{black}X}_t^\epsilon,{\color{black}Y}_t^\epsilon,{\color{black}V}_t^\epsilon)$ moves  to $\mathfrak{s}_-'$
with probability  $1-\alpha({-k_\mus})$
and to $({\color{black}x}_{-k_\mus+1},0,0)$ with probability $\alpha({-k_\mus})$.
The exponential time of jump has mean $\tau^2 \delta^2$.
\item
Otherwise the random dynamics of $({\color{black}X}_t^\epsilon,{\color{black}Y}_t^\epsilon,{\color{black}V}_t^\epsilon)$  is the one of $({\color{black}X}^\delta_t,{\color{black}Y}_t^\delta,{\color{black}V}^\delta_t)$.

\end{itemize}
In this context, the generator ${\cal L}^\epsilon$ of $({\color{black}X}_t^\epsilon,{\color{black}Y}^\epsilon_t,{\color{black}V}^\epsilon_t)$  is
$$
({\cal L}^\epsilon \varphi )(\bz) = 
({\cal L}^{\prime} \varphi )(\bz) \mbox{ for }\bz \in E,\quad \quad
({\cal L}^\epsilon \varphi )(\mathfrak{s}_\pm') = \frac{\varphi(\mathfrak{s}_\pm ) - \varphi(\mathfrak{s}_\pm')}{\epsilon \tau^2 \delta^2}.
$$
Given $f$ a bounded function, we consider the function 
$$
u_\lambda^\epsilon({\color{black}x},{\color{black}y},v;f) = \mathbb{E}_{({\color{black}x},{\color{black}y},v)} \Big[ \int_0^\infty e^{-\lambda s} f({\color{black}X}_s^\epsilon,{\color{black}Y}_s^\epsilon,{\color{black}V}_s^\epsilon) \textup{d} s \Big]
$$
which satisfies the equation
\begin{equation}
\label{plambdaepsilon}
\lambda u_\lambda^\epsilon - {\cal L}^\epsilon u_\lambda^\epsilon = f \: \mbox{ in } \: E \cup \{\mathfrak{s}_\pm'\} .
\end{equation}

We want to establish the representation formula \eqref{repformula}.
The function $f$ is arbitrary and can be decomposed as a sum of two functions: one symmetric $f_s = \frac{1}{2}(f+f \circ \gamma)$ and one antisymmetric $f_a = \frac{1}{2}(f-f \circ \gamma)$ where $\forall ({\color{black}x},{\color{black}y},v) \in E$, $\gamma({\color{black}x},{\color{black}y},v) = -({\color{black}x},{\color{black}y},v)$ and $\gamma(\mathfrak{s}_\pm')=\mathfrak{s}_\mp'$.
We first show that we have the representation formula
\begin{align*}
u_\lambda^\epsilon({\color{black}x},{\color{black}y},v;f) = w_\lambda({\color{black}x},{\color{black}y},v;f) - \nu_\lambda^\epsilon(f)w_\lambda({\color{black}x},{\color{black}y},v;1) 
+ \mu_\lambda^\epsilon(f) \big(({h}_\lambda^+ ({\color{black}x},{\color{black}y},v)- {h}_\lambda^-({\color{black}x},{\color{black}y},v)\big) + \frac{{\pi}_\lambda^\epsilon(f)}{\lambda}   ,
\end{align*}
where
\begin{align*}
{\pi}_\lambda^\epsilon(f) &= \frac{w_\lambda(\mathfrak{s}_+;f)+w_\lambda(\mathfrak{s}_-;f) + \epsilon \tau^2 \delta^2(f(\mathfrak{s}_+) + f(\mathfrak{s}_-)) }{2 w_\lambda(\mathfrak{s}_+;1) + 2 \epsilon \tau^2 \delta^2  
} ,
\\
\mu_\lambda^\epsilon(f) &= \frac{w_\lambda(\mathfrak{s}_+;f)- w_\lambda(\mathfrak{s}_-;f) + \epsilon \tau^2 \delta^2(f(\mathfrak{s}_+) - f(\mathfrak{s}_-))}{2 (1-{h}_\lambda^+(\mathfrak{s}_+) + {h}_\lambda^-(\mathfrak{s}_+) + \epsilon \tau^2 \delta^2 
)}.
\end{align*}
We split the proof into two parts.

{\it Step 1.
Assume $f$ is symmetric. We have
\begin{align}
{\pi}_\lambda^\epsilon(f) &=\frac{w_\lambda(\mathfrak{s}_+,f) + \epsilon \tau^2 \delta^2 f(\mathfrak{s}_+')}{w_\lambda(\mathfrak{s}_+,1) + \epsilon \tau^2 \delta^2 
} ,
\\
\label{step1-prop32}
u_\lambda^\epsilon({\color{black}x},{\color{black}y},v;f) 
&=
w_\lambda({\color{black}x},{\color{black}y},v;f) + {\pi}_\lambda^\epsilon(f) \Big( \frac{1}{\lambda} - w_\lambda({\color{black}x},{\color{black}y},v;1) \Big).
\end{align}
}
{\it Proof of step1.}
By linearity of the operator ${\cal L}^{\prime}$, 
it is clear that $w_\lambda(\cdot;f) + {\pi}_\lambda^\epsilon(f) \left ( \frac{1}{\lambda} - w_\lambda(\cdot;1) \right )$ satisfies \eqref{plambdaepsilon} in $E$.
Moreover, by definition of the constant ${\pi}_\lambda^\epsilon(f)$ and by linearity of the operator ${\cal L}^\epsilon$, the equation is also satisfied in $\{\mathfrak{s}_\pm'\}$.
Since $\lambda>0$, the solution of \eqref{plambdaepsilon} is unique and thus \eqref{step1-prop32} is shown.

{\it Step 2.
Assume $f$ is antisymmetric. We have
\begin{align*}
\mu_\lambda^\epsilon(f) &=
\frac{w_\lambda(\mathfrak{s}_+,f) + \epsilon \tau^2 \delta^2 f(\mathfrak{s}_+')}{1-{h}_\lambda^+(\mathfrak{s}_+)+{h}_\lambda^-(\mathfrak{s}_+) + \epsilon \tau^2 \delta^2 
}  , \\
u_\lambda^\epsilon({\color{black}x},{\color{black}y},v;f) 
&=
w_\lambda({\color{black}x},{\color{black}y},v;f) + \mu_\lambda^\epsilon(f) \big( {h}_\lambda^+ ({\color{black}x},{\color{black}y},v)- {h}_\lambda^-({\color{black}x},{\color{black}y},v) \big) .
\end{align*}
}
{\it Proof of step 2.}
The proof follows the same logic to what is done in step 1 except that we replace the function $w_\lambda(\cdot;f) + {\pi}_\lambda^\epsilon(f) \left ( \frac{1}{\lambda} - w_\lambda(\cdot;1) \right )$
by $w_\lambda(\cdot;f) + \mu_\lambda^\epsilon(f) \left ( {h}_\lambda^+ - {h}_\lambda^- \right )$ and the constant ${\pi}_\lambda^\epsilon(f)$ by $\mu_\lambda^\epsilon(f)$. 

{\it Step 3.}
To treat the general case of $f$, we collect what was done in the two previous steps:
\begin{align*}
u_\lambda^\epsilon(\cdot;f) & = u_\lambda^\epsilon(\cdot;f_s) + u_\lambda^\epsilon(\cdot;f_a)\\
& = w_\lambda(\cdot;f_s) + {\pi}_\lambda^\epsilon(f_s) \Big( \frac{1}{\lambda} - w_\lambda(\cdot;1) \Big) + w_\lambda(\cdot;f_a) + \mu_\lambda^\epsilon(f_a) \left ( {h}_\lambda^+ - {h}_\lambda^- \right ) \\
& = w_\lambda(\cdot;f) + {\pi}_\lambda^\epsilon(f) \Big( \frac{1}{\lambda} - w_\lambda(\cdot;1) \Big) + \mu_\lambda^\epsilon(f) 
\big( {h}_\lambda^+ - {h}_\lambda^- \big).
\end{align*}
We finally get the  representation formula (\ref{repformula}) from the fact that
$$
u_\lambda^\epsilon({\color{black}x},{\color{black}y},v;f) \to u_\lambda({\color{black}x},{\color{black}y},v;f) = \mathbb{E}_{({\color{black}x},{\color{black}y},v)}\Big[ \int_0^\infty e^{-\lambda s} f({\color{black}X}_s^\delta,{\color{black}Y}_s^\delta,{\color{black}V}_s^\delta) \textup{d} s\Big]  \: \mbox{as} \: \epsilon \to 0,
$$
and
$
\lim \limits_{\lambda \downarrow 0} \lambda u_\lambda(f) 
=  \lim \limits_{\lambda \downarrow 0} \lim \limits_{\epsilon \downarrow 0} \lambda u_\lambda^\epsilon(f)$.

{\color{black}
\section{Proof of the integrability of $t \mapsto \EE_{\pi} [ {\color{black}V}^\delta_0 {\color{black}V}^\delta_t]$}
\label{app:intpsi}%
We can decompose
\begin{align}
\EE_\pi [ {\color{black}V}^\delta_t  {\color{black}V}^\delta_0 ]
= \EE_\pi [ {\color{black}V}^\delta_t  {\color{black}V}^\delta_0 {\bf 1}_{t \leq \tau_1} ] + \EE_\pi [ {\color{black}V}^\delta_t  {\color{black}V}^\delta_0{\bf 1}_{t>\tau_1} ] .
\label{eq:decEVV}
\end{align}
The first term of the right-hand side is integrable since
$$
\big|\EE_\pi [ {\color{black}V}^\delta_t  {\color{black}V}^\delta_0 {\bf 1}_{t \leq \tau_1} ] \big| \leq v_{\rm max}^2 \PP_\pi(t \leq \tau_1) ,
$$
and $\int_0^\infty \PP_\pi(t \leq \tau_1)\textup{d}  t =\EE_\pi[\tau_1]<+\infty$.
Denoting $\psi(t) = \EE_{\mathfrak{s}_+}[ {\color{black}V}^\delta_t ]$ (which is such that $\EE_{\mathfrak{s}_-}[ {\color{black}V}^\delta_t ] = -\psi(t)$), and using the strong Markov property, the second term of the right-hand side of (\ref{eq:decEVV}) can be written as
\begin{align*}
 \EE_\pi [ {\color{black}V}^\delta_t  {\color{black}V}^\delta_0{\bf 1}_{t>\tau_1} ]
 =
\EE_\pi [ {\color{black}V}^\delta_0 
{\bf 1}_{t > \tau_1} \psi(t-\tau_1) \big( {\bf 1}_{\bZ_{\tau_1}=\mathfrak{s}_+}
-{\bf 1}_{\bZ_{\tau_1}=\mathfrak{s}_-}\big) ] .
\end{align*}
It is sufficient to show that $\psi$ is integrable in order to complete the proof because then
$$
\int_0^\infty \big|  \EE_\pi [ {\color{black}V}^\delta_t  {\color{black}V}^\delta_0{\bf 1}_{t>\tau_1} ] \big| \textup{d}  t \leq  v_{\rm max} \int_0^\infty |\psi(t) |\textup{d} t  ,
$$
so that $t \mapsto \EE_\pi [ {\color{black}V}^\delta_t  {\color{black}V}^\delta_0  ] $ is integrable by (\ref{eq:decEVV}).

We have, using again the strong Markov property
\begin{align*}
\psi(t) = \EE_{\mathfrak{s}_+}[ {\color{black}V}^\delta_t   {\bf 1}_{t \leq \tau_1}]+\EE_{\mathfrak{s}_+} \big[ \psi(t-\tau_1){\bf 1}_{t>\tau_1}  
 \big( {\bf 1}_{\bZ_{\tau_1}=\mathfrak{s}_+}-{\bf 1}_{\bZ_{\tau_1}=\mathfrak{s}_-}\big) \big] .
\end{align*}
We denote $\kappa(u) = \EE_{\mathfrak{s}_+}  [ {\bf 1}_{\bZ_{\tau_1}=\mathfrak{s}_+}-{\bf 1}_{\bZ_{\tau_1}=\mathfrak{s}_-} |\tau_1=u]$. It satisfies $|\kappa(u)|<1$ for all $u>0$ because $\PP_{\mathfrak{s}_+} ( \bZ_{\tau_1}=\mathfrak{s}_+ | \tau_1=u)\in (0,1)$ and we have
\begin{align*}
\psi(t) = \EE_{\mathfrak{s}_+}[ {\color{black}V}^\delta_t   {\bf 1}_{t \leq \tau_1}]+\EE_{\mathfrak{s}_+} \big[ \psi(t-\tau_1) {\bf 1}_{t>\tau_1} 
\kappa(\tau_1) \big] .
\end{align*}
For any $T>0$, 
\begin{align*}
\int_0^T |\psi(t)|\textup{d} t  &\leq \int_0^T  \EE_{\mathfrak{s}_+}[ |{\color{black}V}^\delta_t |  {\bf 1}_{t \leq \tau_1}] \textup{d}  t +\EE_{\mathfrak{s}_+} \Big[ \int_0^{(T-\tau_1)_+} |\psi(t)| \textup{d}  t  
|\kappa(\tau_1)| \Big]\\
&\leq v_{\rm max} \int_0^{+\infty} \PP_{\mathfrak{s}_+}( \tau_1\geq t) \textup{d} t +\EE_{\mathfrak{s}_+} [|\kappa(\tau_1)|] \int_0^T |\psi(t)|\textup{d} t  .
\end{align*}
Since $\EE_{\mathfrak{s}_+}[|\kappa(\tau_1)|] <1$ and $\int_0^{+\infty} \PP_{\mathfrak{s}_+}( \tau_1\geq t) \textup{d} t = \EE_{\mathfrak{s}_+}[\tau_1]<\infty$, 
this shows that $\psi$ is integrable:
 $$
\int_0^\infty |\psi(t)|\textup{d} t \leq \frac{ v_{\rm max}\EE_{\mathfrak{s}_+}[\tau_1]}{1-\EE_{\mathfrak{s}_+}[|\kappa(\tau_1)|] }.
$$
}


\begin{thebibliography}{9}
\bibitem{cpam_long}
J. Garnier and L. Mertz,
A control variate method driven by diffusion approximation,
arXiv:1906.01225 [this is the extended version of the paper \cite{cpam} with the appendices].
\end{thebibliography}
\end{document}